\documentclass[10pt,reqno]{amsart}

\usepackage{amsfonts,amssymb,amsmath,mathrsfs}
\usepackage{float}
\usepackage{stmaryrd}
\usepackage{srcltx}
\usepackage{epsfig}
\usepackage{graphicx}
\usepackage{caption}
\usepackage{color}
\usepackage{verbatim}
\usepackage{kotex}
\usepackage{hyperref}

\title[ ]{flocking formation and stabilizer of boosted cooperative control on  a sphere}

\author[Sun-Ho Choi]{Sun-Ho Choi}
\address[Sun-Ho Choi]{Department of Applied Mathematics and the Institute of Natural Sciences, Kyung Hee University, 1732 Deogyeong-daero, Giheung-gu, Yongin 17104, Republic of Korea}
\email{sunhochoi@khu.ac.kr}

\author[Dohyun Kwon]{Dohyun Kwon}
\address[Dohyun Kwon]{Department of Mathematics, University of Wisconsin-Madison, 480 Lincoln Dr., Madison, WI 53706, USA}
\email{dkwon7@wisc.edu}

\author[Hyowon Seo]{Hyowon Seo}
\address[Hyowon Seo]{Department of Applied Mathematics and the Institute of Natural Sciences, Kyung Hee University, 1732 Deogyeong-daero, Giheung-gu, Yongin 17104, Republic of Korea}
\email{hyowseo@gmail.com}

\begin{document}

\newtheorem{theorem}{Theorem} [section]
\newtheorem{maintheorem}{Theorem}
\newtheorem{lemma}[theorem]{Lemma}
\newtheorem{proposition}[theorem]{Proposition}
\newtheorem{remark}[theorem]{Remark}
\newtheorem{example}[theorem]{Example}
\newtheorem{exercise}{Exercise}
\newtheorem{definition}{Definition}[section]
\newtheorem{corollary}[theorem]{Corollary}


\newcommand{\noi}{\noindent}
\newcommand{\Z}{\mathbb{Z}}
\newcommand{\R}{\mathbb{R}}
\newcommand{\C}{\mathbb{C}}
\newcommand{\T}{\mathbb{T}}
\newcommand{\bul}{\bullet}
\newcommand{\E}{\mathcal{E}}
\newcommand{\N}{\mathcal{N}}
\newcommand{\RR}{\mathcal{R}}
\newcommand{\D}{\mathcal{D}}
\newcommand{\HH}{\mathcal{H}}

\newcommand{\al}{\alpha}
\newcommand{\dl}{\delta}
\newcommand{\Dl}{\Delta}
\newcommand{\eps}{\varepsilon}
\newcommand{\kk}{\kappa}
\newcommand{\g}{\gamma}
\newcommand{\G}{\Gamma}
\newcommand{\ld}{\lambda}
\newcommand{\lam}{\lambda}
\newcommand{\Ld}{\Lambda}
\newcommand{\s}{\sigma}
\newcommand{\ft}{\widehat}
\newcommand{\wt}{\widetilde}
\newcommand{\cj}{\overline}
\newcommand{\dx}{\partial_x}
\newcommand{\dt}{\partial_t}
\newcommand{\dd}{\partial}
\newcommand{\invft}[1]{\overset{\vee}{#1}}
\newcommand{\lrarrow}{\leftrightarrow}
\newcommand{\embeds}{\hookrightarrow}
\newcommand{\LRA}{\Longrightarrow}
\newcommand{\LLA}{\Longleftarrow}

\newcommand{\wto}{\rightharpoonup}

\newcommand{\jb}[1]
{\langle #1 \rangle}

\newcommand{\dk}[1]{{\color{caribbean green}[#1]}}


\renewcommand{\theequation}{\thesection.\arabic{equation}}
\renewcommand{\thetheorem}{\thesection.\arabic{theorem}}
\renewcommand{\thelemma}{\thesection.\arabic{lemma}}
\newcommand{\bbr}{\mathbb R}
\newcommand{\bbz}{\mathbb Z}
\newcommand{\bbn}{\mathbb N}
\newcommand{\bbs}{\mathbb S}
\newcommand{\bbp}{\mathbb P}
\newcommand{\ddiv}{\textrm{div}}
\newcommand{\bn}{\bf n}
\newcommand{\rr}[1]{\rho_{{#1}}}
\newcommand{\thh}{\theta}
\def\charf {\mbox{{\text 1}\kern-.24em {\text l}}}
\renewcommand{\arraystretch}{1.5}

\thanks{
}

\begin{abstract}
The purpose of this paper is to consider flocking formations of a second order dynamic system on a sphere with a Cucker-Smale type flocking operator and cooperative control. The flocking operator consists of a weighted control parameter and a natural relative velocity. The cooperative control law is given by a combination of attractive and repulsive forces. We prove that the solution to this system converges to an asymptotic spherical configuration depending on the control parameters of the cooperative control law. All possible asymptotic configurations for this system are classified and sufficient conditions for rendezvous, deployment, and local deployment, respectively, are presented. In addition, we provide several numerical simulations to confirm our analytic results. We numerically verify that the solution to this system converges to a formation flight with a nonzero constant speed if we add a boost term. The flocking operator acts as a stabilizer on the spherical cooperative control laws.
\end{abstract}


\maketitle


%
%
\section{Introduction}\label{sec1}
\setcounter{equation}{0}
Simultaneous control of multiple agents attracts much attention of researchers \cite{G-C-F,K-R,M}. With the recent development of hardware, it is necessary to consider a multi-agent control system on the entire Earth's surface \cite{B-M, L-S,M-G,S-S-Z-T}, not just a local range.  The spherical surface has a nonzero curvature rather than a flat space, which yields mathematical challenges. The control system using cooperative manipulation on various manifolds is considered in many studies, such as on a sphere \cite{L2, P-D,T-M-B-G}, Stiefel manifold \cite{M-T-G}, and directed topology \cite{S-M-Z-H-H,Z-X-H-M-T}. In particular, the controllers or algorithms for the uniform deployment of homoclinic agents are investigated using single-integrator model \cite{La-S,P-P} or double-integrator model \cite{L-C}. In \cite{La-S}, the exponential asymptotic stability is shown  using invariant manifold techniques for the single-integrator model. After the pioneering work in \cite{O}, the double-integrator multi-agent system of control has been  extensively studied. However, as the authors in \cite{S} mentioned, the complete characterization of global convergence of the double-integrator system is unsolved and there are still many open problems. The global convergence property over undirected graphs on the $n$-sphere is studied in \cite{M-G,M-T-G2}. The convergence rate problem over topology networks is investigated in \cite{L-C} with  the convergence condition for consensus.


The collective motion of the multi-agent system, including flocking, is a phenomenon frequently observed in nature \cite{Alt, F-E,H-C, S-H}. Recently, many researchers have studied  this phenomenon  because of its wide applicability. Related topics are also being used in various fields \cite{C-F-R-T, D-M1,Kuromoto,T-T,T-B,V-C-B-C-S,Wi}. In particular, the Cucker-Smale model is often studied due to its high utilization, simple structure, and various mathematical phenomena. The basic Cucker-Smale model is the second-order system of ordinary differential  equations (ODEs)  describing the ensemble of multiple agents in  $\bbr^d$, and the control law is composed of the sum of the following weighted internal relaxation forces \cite{C-S2}:
\begin{align*}
\begin{aligned}
\frac{dx_i}{dt} &= v_i, \\
\frac{dv_i}{dt} &= \sum_{j=1}^{N} \frac{\psi_{ij}}{N}(v_j - v_i),
\end{aligned}
\end{align*}
where $x_i$ and $v_i$  are the position  and velocity in $\bbr^d$ of the $i$th agent, respectively, and  $\psi_{ij}$ is the communication rate between the $i$th and $j$th agents.

Recently, not only collective motion in the flat space, but also this phenomenon in various manifolds has been attracting attention from researchers.  In particular, the dynamics on a spherical surface are closely related to our reality. For example, if we consider the movement of migratory birds, the movement of each agent is close to the movement on a spherical surface rather than a flat space of $\bbr^n$.  From an engineering point of view, in recent years, battery performance has rapidly improved, and with the development of an aerial vehicle using sunlight, the movement of unmanned aerial vehicles is approaching a global scale \cite{R, S-W-X}. With the development of hardware technologies, the need to consider the entire Earth as a spatial background has increased.


In this paper, we deal with the second-order system of ODEs on $\mathbb{S}^2$ to achieve a formation flight that maintains the desired pattern while multiple agents fly at a constant speed on a spherical surface.
\begin{align}
\begin{aligned}\label{main}
\dot{x}_i&=v_i,\quad i =1, \ldots, N,\\
\dot{v}_i&=-\frac{\|v_i\|^2}{\|x_i\|^2}x_i+\sum_{j=1}^N\frac{\psi_{ij}}{N}\big(R_{x_j \shortrightarrow x_i}(v_j)-v_i\big)+\sum_{j=1,j\ne i}^N \frac{\sigma_{ij} }{N}(\|x_i\|^2x_j - \langle x_i,x_j \rangle  x_i),
\end{aligned}
\end{align}
where   $x_i$ and $v_i$ are the position and velocity of the $i$th agent located on  $\mathbb{S}^2$ and the flocking operator  on  $\mathbb{S}^2$ is expressed by the relative velocity $R_{x_j \shortrightarrow x_i}(v_j)-v_i$.  The rotation operator $R_{x_1\shortrightarrow x_2}(y)$, introduced in  \cite{C-K-S}, is given by
\begin{align*}
R_{x_1\shortrightarrow x_2}(y)=R(x_1, x_2)\cdot y,
\end{align*}
where $y$ is a column vector. For  column vectors $x_1,x_2\in \mathbb{S}^2$  with $x_1\ne -x_2$, $R(x_1,x_2)$ is the following $3\times 3$ matrix.
  \begin{align}\nonumber R(x_1,x_2) :=\left\{ \begin{aligned}
  &\langle x_1 , x_2\rangle I - x_1 x_2^T + x_2 x_1^T + (1-  \langle x_1, x_2\rangle) \Big( \frac{x_1 \times x_2}{|x_1 \times x_2|} \Big) \Big( \frac{x_1 \times x_2}{|x_1 \times x_2|} \Big)^T,\quad&\mbox{if}\quad x_1\ne x_2,\\
  &I,\quad&\mbox{if}\quad x_1=x_2,
  \end{aligned}
  \right.\end{align}
where $I$ is the identity matrix in $\mathbb{R}^3$ and $M^T$ is the transpose of a matrix $M$.

The last term in the right-hand side of \eqref{main} is the cooperative control law on  $\mathbb{S}^2$ that is expressed as an appropriate combination of well-known attractive and repulsive forces between $i$th and $j$th nodes  based on the position of the agents \cite{L-S,PKH10} and $\sigma_{ij}$ is the  inter-particle force parameter between $i$th and $j$th agents given by
\begin{align}\label{sigma}
\sigma_{ij}=\sigma(\|x_i-x_j\|^2)
\end{align}
and
\begin{align}\label{sigma_def}
\sigma(x)=\sigma_a-\frac{\sigma_r}{x}, \quad\mbox{for nonnegative real numbers}~ \sigma_a,\sigma_r\geq 0.
\end{align}
Unlike \cite{L-S}, our model includes the flocking operator instead of the damping term. Therefore, we can add a boost term to our model for the formation flight as in \eqref{eqb}. See Section \ref{sec4}. We further assume that  $\psi_{ij}$ is the communication rate between $i$th and $j$th nodes and
\begin{align}
\label{eqn:psi_ij}
\psi_{ij}=\psi(\|x_i - x_j\|),\end{align}
and it satisfies
\begin{align}
\label{eqn:psi}
\mbox{$\psi$ is a nonnegative decreasing $C^1$ function  on $[0,2]$ with $\psi(2) = 0$ and  $\psi'(2)<0$.}
\end{align}
Throughout this paper we suppose the following admissible initial data conditions: for all $i\in \{1,\ldots,N\}$,
\begin{align}
\label{eqn:ini}\langle x_i(0),x_i(0)\rangle=1 \quad \mbox{and}\quad \langle x_i(0),v_i(0)\rangle=0. \end{align}
From the flocking operator on our system, we expect that the solution to \eqref{main} has a velocity alignment property. See \cite{C-K-S,C-K-S2} for the motivation and the detailed properties of the above operators.


Our system consists of the cooperative control law and  Cucker-Smale type flocking operator. The cooperative control term is given by an appropriate linear combination of attractive and repulsive forces between the agents in the ensemble. It is known by \cite{L-S} that the cooperative control law combining the damping term can realize steady-state formations on a spherical surface. Due to the damping term, the ensemble of all agents is stabilized, and one can obtain a stationary formation on a sphere. In our case, to implement the formation flight with a nonzero constant speed, one may replace the damping term with the boost term. However, the control design without the damping term does not produce a robust pattern and creates a chaotic ensemble. We observe that the movement of the ensemble becomes very unstable. See Figures \ref{fig0} and \ref{fig11} in Section \ref{sec4} for more details and numerical simulations.

Using flocking and boost terms, we design a control law to obtain the desired stable formation flight on a spherical surface with a nonzero constant speed. Here, the flocking operator acts as a stabilizer for this boost term so that agents in the ensemble stably maintain their formation and fly  at a predetermined constant speed. By changing the cooperative control term like \cite{L-S}, various formations can be formed. In what follows, we completely classify these types of limit cycles in the maximal invariant set of this system and provide a rigorous mathematical proof for the corresponding convergence result. Due to the geometric constraint of the spherical surface and the special properties of the flocking operator, the interaction between the cooperative control and the flocking operator allows the agents to converge to various pattern formations even without the damping term.

\begin{definition}
\cite{C-K-S}
If the solution $\{(x_i(t),v_i(t))\}_{i=1}^N $ to \eqref{main} subject to the admissible initial condition in  \eqref{eqn:ini} satisfies the following condition:\begin{align}\nonumber
\max_{1\leq i,j\leq N} \|x_i(t) + x_j(t)\| \|  R _{x_j(t)\shortrightarrow x_i(t)}(v_j(t)) - v_i(t) \| = 0,
\end{align}
then  we say that the system has velocity alignment at $t\geq 0$.

Similarly, we define asymptotic velocity alignment at $t=\infty$, if
\begin{align}\nonumber
\lim_{t\to \infty}\max_{1\leq i,j\leq N} \|x_i(t) + x_j(t)\| \|  R _{x_j(t)\shortrightarrow x_i(t)}(v_j(t)) - v_i(t) \| = 0.
\end{align}
\end{definition}

Additionally, by the geometrically constrained control law in \eqref{main}, we consider the following three types of asymptotic behaviors.

\begin{definition}Let $\{(x_i(t),v_i(t))\}_{i=1}^N $ be the solution to \eqref{main}.
\begin{enumerate}
\item
The ensemble $\{(x_i(t),v_i(t))\}_{i=1}^N $ has an asymptotic rendezvous, if
\[\lim_{t\to\infty }\max_{1\leq i,j\leq N}\|x_i(t)-x_j(t)\|=0.\]
\item
The ensemble $\{(x_i(t),v_i(t))\}_{i=1}^N $ has an asymptotic formation configuration, if
\[\lim_{t\to \infty}\rho(t)=r_0>0.\]
\item
The ensemble $\{(x_i(t),v_i(t))\}_{i=1}^N $  has an asymptotic uniform deployment, if
\[\lim_{t\to\infty }\bar{x}(t)=0,\]
where $\bar{x}$ is the centroid of the agents and  $\rho$ is the configuration measurement  given by
\[\bar{x}(t)=\frac{1}{N}\sum_{i=1}^N x_i(t),\]
 and\[\rho(t)=\sum_{i=1}^N\left\|x_i(t)-\frac{\langle \bar{x},x_i \rangle}{\|\bar{x}\|^2}\bar{x}(t)\right\|^2,\quad \mbox{ if
}~\|\bar{x}(t)\|\ne 0.\]
\end{enumerate}
\end{definition}

\begin{remark}
\begin{enumerate}
\item If we consider an axis  parallel  to $\bar{x}$ passing through to the origin, then $\rho$ is the total sum of squared distances between $x_i$ and the axis.

\item The structure of $\sigma(\cdot)$ can be used to obtain a pattern scale measurement  $\rho(\cdot)$ as in \cite{L-S}.

\end{enumerate}

\end{remark}

 \begin{definition}For a given phase state vectors $x_i\in\bbs^2$ and $v_i\in T_{x_i}\bbs^2\subset \bbr^3$ with $\langle x_i,v_i \rangle=0$, we let
  \[X=\{x_i\}_{i=1}^N, \quad V=\{v_i\}_{i=1}^N,\quad \mbox{and}\quad Z=\{(x_i,v_i)\}_{i=1}^N.\]
  We define  an energy functional $\E(Z)$ by
\begin{align*}
\E(Z)&= \E_K(V) + \E_C(X)-\E_C^{min}, \end{align*}
where
\begin{align*}
\E_K(V)&:= \frac{1}{2N}\sum_{i=1}^N\|v_i\|^2,
\\
 \E_C(X)& :=  \frac{\sigma_a}{4N^2 } \sum_{i,j=1}^N \|x_i-x_j\|^2-\frac{\sigma_r}{4N^2 } \log\prod_{i,j=1, i\ne j}^N\|x_i-x_j\|^2,
\end{align*}
and
\[\E_C^{min}= \min_{ X\in (\bbs^2)^N}\E_C(X).\]

 \end{definition}

We next present our main analytic results in this paper.
\begin{maintheorem}\label{thm:1}
Assume that $\psi_{ij}$ satisfies  \eqref{eqn:psi_ij}-\eqref{eqn:psi} and the initial data $\{(x_i(0),v_i(0))\}_{i=1}^N$ satisfy the admissible initial data conditions in \eqref{eqn:ini}. If $\sigma_r= 0$, then there exists a unique global solution  $\{(x_i(t),v_i(t))\}_{i=1}^N $ to \eqref{main}.

If $\sigma_r > 0$ and the initial data satisfy \eqref{eqn:ini} and    $x_i(0)\ne x_j(0)$ for any $1\leq i\ne j\leq N$, then there exists a unique global solution  $\{(x_i(t),v_i(t))\}_{i=1}^N $ to \eqref{main} and for any $t>0$ and $1\leq i\ne j\leq N$,
 \[x_i(t)\ne x_j(t).\]
Moreover, for any initial data, the ensemble has velocity alignment at $t=\infty$.
\end{maintheorem}

\begin{maintheorem}\label{thm:2} Let $N\geq 3$ and  $Z(t)=\{(x_i(t),v_i(t))\}_{i=1}^N $  be the solution  to \eqref{main} with \eqref{eqn:ini}. Additionally, we suppose that      $x_i(0)\ne x_j(0)$ for any $1\leq i\ne j\leq N$ when the repulsive force parameter is nonzero: $\sigma_r > 0$.
Then the solution $Z(t)=\{(x_i(t),v_i(t))\}_{i=1}^N $  satisfies the following asymptotic behaviors.

\begin{enumerate}
  \item[(i)] If $\sigma_a>0$, $\sigma_r =0$ and $\displaystyle \E(Z(0))<\E_C^0=\frac{\sigma_a(N-1)}{N^2 }$, then the solution  has an asymptotic rendezvous.
  \item[(ii)]
If $\displaystyle  \frac{2N}{N-1}\sigma_a>\sigma_r>0$ and $\E(Z(0))<\E_{C}^1$, then the ensemble  has an asymptotic formation configuration, where  $\E_{C}^1$ is the minimum energy of stationary solutions given by
 \[\qquad \qquad\E_{C}^1:=\inf_{\|x_i\|=1, \bar{x}=0}\bigg\{\E_C(X)-\E_C^{min} : \sum_{j=1,j\ne i}^N \frac{\sigma_{ij} }{N}(\|x_i\|^2x_j - \langle x_i,x_j \rangle  x_i)=0,\quad 1\leq i\leq N \bigg\} .\]

  \item[(iii)] If $\displaystyle\sigma_r \geq \frac{2N}{N-1}\sigma_a> 0$ or  $\displaystyle\sigma_r>0=\sigma_a$, then  $\{(x_i(t),v_i(t))\}_{i=1}^N$ has an asymptotic uniform deployment.

\end{enumerate}

\end{maintheorem}

\begin{remark}\label{rmk 1.2}
\begin{enumerate}
\item The conditions for the above theorem are almost optimal. Without the above sufficient conditions, the result in Theorem \ref{thm:2} cannot be established,  since there are   unstable steady-states that do not satisfy the initial constraints on energy. Further, in the case of $N=2$, we cannot  expect that the above results hold due to the result in Corollary \ref{propN2}.
\item The $N$-dependence of $\E_C^0$ in  Theorem \ref{thm:2} $(i)$ can be removed by adding a restriction of the initial configuration size; see \cite{C-K-S2}.

\item The solution to our model converges to a stable-state, but the speed of agents may not converge to $0$. See Corollary \ref{propN2} and Proposition \ref{prop 1.6}. 
                              \end{enumerate}

\end{remark}

The remaining subject is the formation flight of a multi-agent ensemble on a sphere. As mentioned before, by adding a boost term to \eqref{main}, we will implement the formation flight of multiple agents on a sphere  with a nonzero constant speed through numerical simulations. We emphasize that cooperative control law alone cannot implement this nonzero speed formation pattern.  In the formation flight with a nonzero constant speed, the formation of agents is very similar to the formation without the boost term in Theorem \ref{thm:2}. In addition, the mathematical results in Theorem \ref{thm:2} are confirmed through numerical simulation, and various formation patterns are implemented through more general control parameters. We will describe detailed numerical simulation results in Section \ref{sec4}.

The rest of this paper is organized as follows. In Section \ref{sec2}, we present elementary properties of the rotation operator and the global well-posedness of \eqref{main}. In Section \ref{sec3}, we prove the convergence results in our main theorems. In Section \ref{sec4}, we provide several numerical simulations that verify our analytical results and we present the robust formation flight by adding a boost term. Finally, Section \ref{sec5} is devoted to the summary of our main results.  \newline

{\bf Notation:}
Throughout this paper, we use the following notation.

\begin{itemize}
  \item The domain  is a unit sphere defined by
\begin{align*}
\mathbb{S}^2 :=\{(a,b,c)^T\in \bbr^3: a^2+b^2+c^2=1 \}.
\end{align*}
  \item We set $X=\{x_i\}_{i=1}^N$ is a sequence of column vectors in $\bbs^2$, $V=\{v_i\}_{i=1}^N$ is a sequence of the corresponding velocity vectors with $\langle x_i,v_i \rangle=0$, and $Z=\{(x_i,v_i)\}_{i=1}^N$ be a sequence of phase vectors.

  \item We denote the tangent bundle of $\bbs^2$ by $T\bbs^2$.
  \item For a given $z_1,z_2 \in \bbr^3$, we use  $\langle z_1,z_2\rangle$ to denote the standard inner product in $\bbr^3$ and the standard symbol
 \[\|z_1\| = \|z_1\|_2=\sqrt{\langle z_1,z_1\rangle}.\]

\end{itemize}

\section{Elementary properties of the Newtonian dynamics with the flocking operator and its well-posedness}\label{sec2}
\setcounter{equation}{0}
In this section, we briefly review the elementary properties of  \eqref{main} and prove the existence and uniqueness of the global-in-time solution to \eqref{main}. The rotation operator $R_{x_i\shortrightarrow x_j}$  has several remarkable properties such as preserving the modulus of the vector and always having the inverse, but it has singularity when  $x_i$ and $x_j$ are antipodal points. Thus, to prove the existence of a global-in-time solution, we must deal with this singularity. For this, we assumed that $\psi(\cdot)$ satisfies the admissible conditions in \eqref{eqn:psi}. Using this condition, in \cite{C-K-S}, we proved that the rotation operator $\psi R$  is locally Lipschitz. From this result, we can obtain the global existence of a solution  to the model. In addition, we define an energy functional $\E$ and prove that $\E$ is a decreasing function for time $t\geq 0$. From this property and the definition of $\E$, we can obtain  various boundedness of the ensemble which play an important role in later proving the asymptotic behaviors of the solution.

\begin{lemma}\label{lemma 2.3} (Lemma 2.4 in \cite{C-K-S})
Let  $x_1, x_2 \in \mathbb{S}^2$ be not antipodal points $(x_1 \ne -x_2)$. Then the rotation operator $R_{\cdot \shortrightarrow \cdot}(\cdot)$ satisfies
\begin{align}
\nonumber
R_{x_1\shortrightarrow x_2}(x_1) = x_2, \quad R_{x_1\shortrightarrow x_2}(x_2) = 2\langle x_1, x_2\rangle x_2 - x_1~  \hbox{ and }  ~R_{x_1\shortrightarrow x_2} (x_1 \times x_2) = x_1 \times x_2.
\end{align}
Furthermore, we have
\begin{align}
\nonumber
R_{x_1\shortrightarrow x_2}^{T} = R_{x_2\shortrightarrow x_1},\quad  
R_{x_1\shortrightarrow x_2}^T \circ R_{x_1\shortrightarrow x_2} = I_{\mathbb{S}^2}.
\end{align}

\end{lemma}

\begin{lemma}\label{lemma 1.1}
We assume that $\psi_{ij}$ are nonnegative bounded functions for all $i,j\in \{1,\ldots,N\}$.
Let $\{(x_i(t),v_i(t))\}_{i=1}^N$ be a solution to \eqref{main} and the initial data satisfy the admissible conditions in \eqref{eqn:ini}. Then  for all $i\in \{1,\ldots,N\}$ and  $t>0$,
\begin{align*}\langle v_i(t),x_i(t)\rangle=0 \hbox{ and } \quad \|x_i(t)\|= 1.\end{align*}
\end{lemma}
\begin{proof}
We follow the same argument in \cite{C-K-S}. We take the inner product between  the second equation in \eqref{main} and $x_i$ to obtain
\begin{align}\begin{aligned}\label{eq 2.0}
\langle \dot{v}_i, x_i\rangle &=-\|v_i\|^2+\sum_{j=1}^N\frac{\psi_{ij}}{N}\big(\langle R_{x_j\shortrightarrow x_i}(v_j), x_i\rangle-\langle v_i, x_i\rangle\big)
\\&\qquad\qquad\qquad+\sum_{j=1,j\ne i}^N \frac{\sigma_{ij} }{N}\big(\|x_i\|^2\langle x_j, x_i\rangle  - \langle x_i,x_j \rangle \langle x_i,x_i\rangle\big)
\\&=-\|v_i\|^2+\sum_{j=1}^N\frac{\psi_{ij}}{N}\big(\langle R_{x_j\shortrightarrow x_i}(v_j), x_i\rangle-\langle v_i, x_i\rangle\big)
.
\end{aligned}
\end{align}
By Lemma \ref{lemma 2.3}, \eqref{eq 2.0} implies
\begin{align}\begin{aligned}
\label{eq 2.7}
\langle \dot{v}_i, x_i\rangle
&=-\|v_i\|^2+\sum_{j=1}^N\frac{\psi_{ij}}{N}\big(\langle v_j, x_j\rangle-\langle v_i, x_i\rangle\big).
\end{aligned}
\end{align}
The equalities in \eqref{eq 2.7} yield that
\begin{align*}
\frac{d}{dt}\sum_{i=1}^N|\langle v_i,x_i\rangle|^2&=2\sum_{i=1}^N \big( \langle \dot{v}_i,x_i\rangle+\langle v_i,\dot{x}_i\rangle\big)\langle v_i,x_i\rangle  \\
&=  2\sum_{i=1}^N\big(\langle \dot{v}_i, x_i\rangle +\|v_i\|^2\big)~\langle v_i,x_i\rangle
\\
&= 2\sum_{i=1}^N \sum_{j=1}^N\frac{\psi_{ij}}{N}\big(\langle v_j, x_j\rangle-\langle v_i, x_i\rangle\big)~\langle v_i,x_i\rangle
\\
&\leq  2\sum_{i=1}^N \sum_{j=1}^N\frac{\psi_{ij}}{N}\langle v_j, x_j\rangle~\langle v_i,x_i\rangle.
\end{align*}
Then we have
\begin{align*}
\frac{d}{dt}\sum_{i=1}^N |\langle v_i(t),x_i(t)\rangle|^2&\leq2\max_{1\leq j,k\leq N}\psi_{jk} \sum_{i=1}^N \big|\langle v_i(t), x_i(t)\rangle \big|^2.
\end{align*}

From Gronwall's inequality and the admissible conditions on the initial  data, it follows  that
\begin{align*}\sum_{i=1}^N |\langle v_i(t),x_i(t)\rangle|^2\equiv0,\quad \mbox{for}~t>0.\end{align*}
Again by the inner product between $\dot{x}_i$ and $x_i$ and the first equation of \eqref{main}, we have
\begin{align*}
\frac{d}{dt}\| x_i\|^2=2\langle \dot{x}_i,x_i\rangle =2\langle v_i,x_i\rangle.
\end{align*}
Therefore,  we have $\|x_i(t)\|\equiv 1$, for $t>0$, $i\in\{1,\ldots N\}$.

\end{proof}

 \begin{proposition}\label{prop 2.4}
We assume that  $\psi_{ij}$  satisfies  \eqref{eqn:psi_ij}-\eqref{eqn:psi}
 and $\sigma_{ij}$ is defined by \eqref{sigma}-\eqref{sigma_def}. For  a solution $Z(t)=\{(x_i(t),v_i(t))\}_{i=1}^N $ to \eqref{main} satisfying the admissible conditions in \eqref{eqn:ini}, the following energy identity holds.
\begin{align}\label{EDi}
\frac{d\E(Z(t))}{dt}=-\sum_{i,j=1}^N\frac{\psi(\|x_i(t)-x_j(t)\|)}{2N^2}\| R_{x_j \shortrightarrow x_i}(v_j(t))-v_i(t)\|^2.\end{align}
 \end{proposition}
\begin{proof}
By Lemma \ref{lemma 2.3},
\begin{align*}
\sum_{i,j=1}^N\frac{\psi_{ij}}{N}\big\langle R_{x_j \shortrightarrow x_i}(v_j)-v_i, v_i\big \rangle
&=\sum_{i,j=1}^N\frac{\psi_{ij}}{N}\big\langle v_j-R_{x_i \shortrightarrow x_j}(v_i), R_{x_i \shortrightarrow x_j}(v_i)\big \rangle\\
&=-\sum_{i,j=1}^N\frac{\psi_{ij}}{N}\big\langle R_{x_i \shortrightarrow x_j}(v_i)-v_j, R_{x_i \shortrightarrow x_j}(v_i)\big \rangle.
\end{align*}
If we change the role of the indices $i$ and $j$ in the above, then we obtain that
\begin{align*}
\sum_{i,j=1}^N\frac{\psi_{ij}}{N}\big\langle R_{x_j \shortrightarrow x_i}(v_j)-v_i, v_i\big \rangle
&=-\sum_{i,j=1}^N\frac{\psi_{ji}}{N}\big\langle R_{x_j \shortrightarrow x_i}(v_j)-v_i, R_{x_j \shortrightarrow x_i}(v_j)\big \rangle.
\end{align*}
Since we assume that $\psi_{ij}=\psi_{ji}$ for all $i,j\in \{1,\ldots,N\}$,
\begin{align}\label{lemma 1.2}
 \sum_{i,j=1}^N\frac{\psi_{ij}}{N}\big\langle R_{x_j \shortrightarrow x_i}(v_j)-v_i, v_i\big \rangle= - \sum_{i,j=1}^N\frac{\psi_{ij}}{2N}\| R_{x_j \shortrightarrow x_i}(v_j)-v_i\|^2.
\end{align}

We multiply the second equation in \eqref{main} by $v_i$ to obtain
\begin{align*}
\frac{1}{2}\frac{d}{dt} \|v_i\|^2 =\langle \dot v_i,v_i\rangle=&-\frac{\|v_i\|^2}{\|x_i\|^2}\langle x_i,v_i\rangle +\sum_{j=1}^N\frac{\psi_{ij}}{N}\big\langle R_{x_j \shortrightarrow x_i}(v_j)-v_i, v_i\big \rangle\\
&\qquad+\sum_{j=1,j\ne i}^N \frac{\sigma_{ij} }{N}\Big(\|x_i\|^2\langle x_j, v_i\rangle  - \langle x_i,x_j \rangle \langle x_i,v_i \rangle \Big).
\end{align*}
By the properties in Lemma \ref{lemma 1.1},
%
%
%
\begin{align}\label{eq 2.25}
\frac{1}{2}\frac{d}{dt} \sum_{i=1}^N\|v_i\|^2 = \sum_{i,j=1}^N\frac{\psi_{ij}}{N}\big\langle R_{x_j \shortrightarrow x_i}(v_j)-v_i, v_i\big \rangle+\sum_{i,j=1,j\ne i}^N \frac{\sigma_{ij} }{N}\langle x_j, v_i\rangle.
\end{align}

By \eqref{lemma 1.2} and \eqref{eq 2.25}, the following holds.
\begin{align}\begin{aligned}\label{eq 2.26}
\frac{d\E_K}{dt}&=\frac{1}{2N}\frac{d}{dt} \sum_{i=1}^N\|v_i\|^2\\& = -\sum_{i,j=1}^N\frac{\psi_{ij}}{2N^2}\| R_{x_j \shortrightarrow x_i}(v_j)-v_i\|^2+\sum_{i,j=1,j\ne i}^N \frac{\sigma_{ij} }{N^2}\langle x_j, v_i\rangle.
\end{aligned}\end{align}
Since $\sigma_{ij}=\sigma_{ji}$ for all indices $1\leq i\leq N$, we have
 \begin{align}\begin{aligned}\label{eq 2.27}
 \sum_{i,j=1,j\ne i}^N \frac{\sigma_{ij} }{N^2}\langle x_j, v_i\rangle&= \sum_{i,j=1,j\ne i}^N \frac{\sigma_{ij} }{2N^2}(\langle x_j, v_i\rangle+\langle x_i, v_j\rangle)\\
 &= -\sum_{i,j=1,j\ne i}^N \frac{\sigma_{ij} }{2N^2}\langle x_i-x_j, v_i-v_j\rangle
 \\& = -\frac{d\E_C}{dt}
 .
\end{aligned}\end{align}
 From \eqref{eq 2.26}-\eqref{eq 2.27}, we can obtain the  energy dissipation identity in \eqref{EDi}.

\end{proof}

\begin{lemma}
\label{lem:con}(Lemma 3.5 in \cite{C-K-S})
Let
\[Q:= (\R^3 \setminus \{0\}) \times (\R^3 \setminus \{0\}) \times \R^3\]
and   $T(\cdot,\cdot,\cdot) : Q \rightarrow \R^3$ be a function defined by
\begin{align}\label{eqn:1con}
T(x_1,x_2,v) =\left \{
\begin{aligned}
&\frac{\|x_2\|}{\|x_1\|}\psi \Big( \Big\| \frac{x_1}{\|x_1\|} - \frac{x_2}{\|x_2\|}\Big\| \Big) R_{\small \frac{x_2}{\|x_2\|}\shortrightarrow \frac{x_1}{\|x_1\|}}(v),\quad  &\hbox{ if }  \frac{x_1}{\|x_1\|} + \frac{x_2}{\|x_2\|} \neq 0,\\
&0,\quad  &\hbox{ if } \frac{x_1}{\|x_1\|} + \frac{x_2}{\|x_2\|} = 0.
\end{aligned}\right.
\end{align}
 If $\psi$ satisfies assumptions in Theorem \ref{thm:1}, then $T(\cdot,\cdot,\cdot)$ is locally Lipschitz continuous in $Q$.
\end{lemma}
\vspace{1em}

\begin{proof}[\bf Proof of Theorem \ref{thm:1}:] The case of $\sigma_r=0$ is similar to the case of $\sigma_r>0$. We only prove this theorem for the case of $\sigma_r>0$.   We follow \cite{C-K-S} for the local existence.

We consider the following system of ODEs subject to the initial data $\{(x_i(0),v_i(0))\}_{i=1}^N$ satisfying the admissible conditions in \eqref{eqn:ini} and $x_i(0)\ne x_j(0)$ for any $i,j\in \{1,\ldots,N\}$:
\begin{align}
\begin{aligned}\label{eqn:wel11}
\dot{x}_i&=v_i,\\
\dot{v}_i&=-\frac{\|v_i\|^2}{\|x_i\|^2}x_i+\sum_{j=1}^N\frac{1}{N}\big(T(x_i, x_j,v_j)-\overline{\psi}_{ij}v_i\big)+\sum_{j=1,j\ne i}^N \frac{\sigma_{ij} }{N}\big(\|x_i\|^2x_j - \langle x_i,x_j \rangle  x_i\big),
\end{aligned}
\end{align}
where $T(\cdot, \cdot,\cdot)$ is given in \eqref{eqn:1con} and
\[\overline{\psi}_{ij}= \psi \left( \Big\| \frac{x_i}{\|x_i\|} - \frac{x_j}{\|x_j\|} \Big\|\right).\]
By Lemma~\ref{lem:con}, the right-hand side of \eqref{eqn:wel11} is Lipschitz continuous with respect to $\{(x_i,v_i)\}_{i=1}^N$ in a small neighborhood of $\{(x_i(0),v_i(0))\}_{i=1}^N$ in $\R^{6N}$. By the standard theory of ODEs, a local-in-time solution $Z(t)=\{(x_i(t),v_i(t))\}_{i=1}^N$ of \eqref{eqn:wel11} exists.

By the same argument in the proof of  Lemma \ref{lemma 1.1}, we have
\begin{align*}\begin{aligned}
\langle \dot{v}_i, x_i\rangle &=-\|v_i\|^2+\sum_{j=1}^N\frac{\overline{\psi}_{ij}}{N}\left(\frac{\|x_j\|}{\|x_i\|}\Big\langle R_{\frac{x_j}{\|x_j\|}\shortrightarrow \frac{x_i}{\|x_i\|}}(v_j), x_i\Big\rangle-\langle v_i, x_i\rangle\right)
\\
&= -\|v_i\|^2+\sum_{j=1}^N\frac{\overline{\psi}_{ij}}{N}\big(\left\langle v_j, x_j\right\rangle-\langle v_i, x_i\rangle\big)
.
\end{aligned}
\end{align*}
Therefore,
\begin{align*}
\frac{d}{dt}\sum_{i=1}^N|\langle v_i,x_i\rangle|^2
&=  2\sum_{i=1}^N\big(\langle \dot{v}_i, x_i\rangle +\|v_i\|^2\big)~\langle v_i,x_i\rangle
\\
&= 2\sum_{i=1}^N \sum_{j=1}^N\frac{\overline\psi_{ij}}{N}\big(\langle v_j, x_j\rangle-\langle v_i, x_i\rangle\big)~\langle v_i,x_i\rangle
\\
&\leq  2\sum_{i=1}^N \sum_{j=1}^N\frac{\overline\psi_{ij}}{N}\langle v_j, x_j\rangle~\langle v_i,x_i\rangle.
\end{align*}
Similar to the proof of  Lemma \ref{lemma 1.1}, we obtain
\begin{align*}\sum_{i=1}^N |\langle v_i(t),x_i(t)\rangle|^2\equiv0,~\mbox{for}~ t>0.\end{align*}
Thus, the admissible initial  conditions imply that
\begin{align*}\langle v_i(t),x_i(t)\rangle=0 \hbox{ and } \quad \|x_i(t)\|= 1,\end{align*}
and   $Z(t)=\{(x_i(t),v_i(t))\}_{i=1}^N$ is also a solution of \eqref{main}.

To prove the existence of the global-in-time solution to \eqref{main} and its uniqueness, we consider  the maximal interval $I_M=[0,t_{M})$ of existence of a solution to \eqref{main}. From the energy inequality in Proposition~\ref{prop 2.4},
\[\E_K(V(t))\leq \E(Z(t))\leq \E(Z(0)).\]
Thus,
\[\|v_i(t)\|^2\leq 2N\E(Z(0)),\quad \mbox{ for all }\quad i\in\{1,\ldots,N\}.\] Similarly, we have
\[\E_C(X(t))\leq \E(Z(t))\leq \E(Z(0))\] and this implies that  there is $l>0$ such that
 $\{(x_i,v_i)\}_{i=1}^N$ are uniformly bounded and
\[\|x_i(t)-x_j(t)\|>l.\]
This yields that  the local solution on $I_M=[0,t_{M})$ can extend more and we conclude that
 \[t_{M} =  \infty.\]
\end{proof}

\section{Invariance principle and asymptotic behavior of the solution}
\label{sec3}
\setcounter{equation}{0}
In this section, we consider the asymptotic behavior of the solution $Z(t)=\{(x_i(t),v_i(t))\}_{i=1}^N$ to \eqref{main}. The main idea for the proof of Theorem \ref{thm:2} in Section \ref{sec1} is to use LaSalle's invariance principle and classification of the limit cycles in the maximal invariance set. For this, we use the energy functional $\E$ defined in the previous section. To use LaSalle's invariance principle, we need a functional with a negative semi-definite orbital derivative and a corresponding positively compact invariant set $\Omega_M$ consisting phase points $\{(x_i,v_i)\}_{i=1}^N$, where the orbital derivative of the  functional on $\Omega_M$ is zero. The largest invariant set contains the $\omega$-limit set of the ensemble. We can prove that it is a velocity-aligned set. This can be obtained by combining the geometric properties of the spherical surface and the rotation operator $R$. The concept of  the largest invariant set and LaSalle's invariance principle are described as follows.

\begin{definition}\cite{K-G}
For   a given domain $\mathcal{D}$, let $Z(t)\in \mathcal{D}$ be the solution to an autonomous system
\begin{align}\label{system}
  \dot Z=f(Z).
\end{align}
For a given functional $\mathcal{V}(\cdot)$, the derivative of $\mathcal{V}$ along the trajectory of \eqref{system} is defined by
\[\dot{\mathcal{V}}(Z)=  \langle \nabla \mathcal{V}, f(Z)\rangle. \]
The set $\Omega\subset \mathcal{D}$ is said to be a positively invariant set if   $Z_0\in \Omega$ and $t_0\in\bbr$, the following holds.
\[Z(t)\in \Omega, \quad t\geq t_0, \]
where $Z(t)$ is the solution to $\dot{Z}=f(Z)$  subject to $Z(t_0)=Z_0$.
\end{definition}

\begin{theorem}\label{LITHM}[LaSalle's invariance principle] Let $Z(t)$ be the solution to \eqref{system}
 with a locally Lipschitz function $f$ and let $\Omega\subset \mathcal{D}$ be a  positively invariant compact set  with respect to \eqref{system}. Assume that there is a continuously differentiable functional $\mathcal{V}(\cdot)$ defined in the  domain $\mathcal{D}$ and   it satisfies
\[\dot{\mathcal{V}}(Z)\leq 0,\quad Z\in \Omega.\]
Let $E$ be the set of all points satisfying $\dot{\mathcal{V}}(Z)=0$ and $M$ be the largest invariant set in $E$.
  Then $Z(t)$ starting in $\Omega$  tends to the largest positively invariant set $M$ as $t$ goes to $\infty$.
\end{theorem}
\begin{proof}
For the proof of this theorem, see \cite{K-G}.
\end{proof}
We next define the corresponding positively invariant compact set of \eqref{main}:
\begin{definition}
For a given $\E_0>0$,  we define a set
\[\Omega_{\E_0}=\Big\{Z=\{(x_i,v_i)\}_{i=1}^N:  \E(Z)\leq \E_0 \Big\}\subset ( T\bbs^{2})^N.\]
\end{definition}
By the definition of $\E$ and  Proposition \ref{prop 2.4},  $\Omega_{\E_0}$ is a positively invariant  compact set  with respect to  \eqref{main}. Moreover, Proposition \ref{prop 2.4} implies that
\[\dot{\E}=-\sum_{i,j=1}^N\frac{\psi(\|x_i-x_j\|)}{2N^2}\| R_{x_j \shortrightarrow x_i}(v_j)-v_i\|^2\leq 0.\]
Therefore, we directly apply LaSalle's invariance principle to obtain the following proposition.
\begin{proposition}\label{prop 3.2}Let $\{(x_i(t),v_i(t))\}_{i=1}^N $ be a solution to \eqref{main}  with the admissible conditions in \eqref{eqn:ini}. We assume that  $\psi_{ij}$  satisfies  \eqref{eqn:psi_ij}-\eqref{eqn:psi}
 and $\sigma_{ij}$ is defined by \eqref{sigma}-\eqref{sigma_def}.

If the  initial energy is $\E_0$, the all $\omega$-limit set of the solution  $\{(x_i(t),v_i(t))\}_{i=1}^N$ are contained in the largest positively invariant set  of $\Omega_{\E_0}$ satisfying
\[\dot\E=-\sum_{i,j=1}^N\frac{\psi(\|x_i-x_j\|)}{2N^2}\| R_{x_j \shortrightarrow x_i}(v_j)-v_i\|^2=0.\]
\end{proposition}

We next present characteristics of the largest positively invariant set  of $\Omega_{\E_0}$. For this, the geometric properties of the unit sphere and  the rotation operator are crucially used. We start with a specially designed two-agents system.
\begin{proposition}\label{prop 3.3}We assume that $\psi(\cdot)$ is a nonnegative bounded function satisfying \eqref{eqn:psi} and $\sigma(\cdot)$ is defined by \eqref{sigma_def}.
Let $\{(x_1(t),v_1(t)),(x_2(t),v_2(t))\}$ be the solution to
\begin{align}
\begin{aligned}\label{two_eq}
\dot{x}_1&=v_1,\\
\dot{v}_1&=-\frac{\|v_1\|^2}{\|x_1\|^2}x_1+  \frac{n\sigma(\|x_1-x_2\|^2) }{N}(\|x_1\|^2x_2 - \langle x_1,x_2 \rangle  x_1),\\
\dot{x}_2&=v_2,\\
\dot{v}_2&=-\frac{\|v_2\|^2}{\|x_2\|^2}x_2+  \frac{(N-n)\sigma(\|x_1-x_2\|^2) }{N}(\|x_2\|^2x_1 - \langle x_1,x_2 \rangle  x_2),
\end{aligned}
\end{align}
where $N$ and $n$ are natural numbers with $n\leq N$.
If  $x_1(t)\ne x_2(t)$, $x_1(t)\ne -x_2(t)$ and
\begin{align}\label{eq vel_align}
0=\psi(\|x_1(t)-x_2(t)\|)\| R_{x_2(t) \shortrightarrow x_1(t)}(v_2(t))-v_1(t)\|^2
\end{align} for all $t\geq 0$,  then the solution satisfies one of the followings:
\begin{itemize}
\item[(1)]the two nodes move around  the same great circle at the same constant speed. Their relative position is   antipodal or $\sigma(\|x_1-x_2\|^2)=0$ holds,
\item[(2)]The trajectories of the two nodes form two parallel circles that have the same radius  with distance $\displaystyle 2\sin\frac{\theta}{2}$, where $\theta$ is a constant between $0$ and $\pi$. Moreover, their  speed is
\[\sqrt{\frac{-\sigma(\|x_1-x_2\|^2)}{2}(1+\cos \theta)}.\]

\end{itemize}
\end{proposition}
\begin{proof}
%

We assume that $x_1(t)$ and $x_2(t)$ are not antipodal for any $t\geq 0$. Using a coordinate change, it suffices to consider
 \begin{align}\label{eq 3.2}
   x_1(0)=(1,0,0)^T,\quad v_1(0)=(0,a,b)^T
 \end{align}
 and
 \begin{align}\label{eq 3.3}
x_2(0)=(\cos \theta,\sin \theta,0)^T,\quad v_2(0)=(-a \sin\theta,a\cos\theta,b)^T,
\end{align}
where $\theta$ is a constant satisfying $ 0<\theta<\pi$.

Note that $\psi(\|x_1(t)-x_2(t)\|)\ne 0$ for any $t\geq $, since $x_1(t)$ and $x_2(t)$ are not antipodal for any $t\geq 0$.  By the assumption \eqref{eq vel_align},
\begin{align}\label{eq 3.1}
R_{x_2(t) \shortrightarrow x_1(t)}(v_2(t))=v_1(t),\quad R_{x_1(t) \shortrightarrow x_2(t)}(v_1(t))=v_2(t).
\end{align}
 Since $(x_1(t),v_1(t))$ and $(x_2(t),v_2(t))$ satisfy \eqref{eq 3.1} for all $t\geq 0$, we have
 \[\frac{d}{dt}[R_{x_2(t) \shortrightarrow x_1(t)}(v_2(t))-v_1(t)]=0.\]
Therefore, we have
 \begin{align}\label{eq 3.5}\frac{d R_{x_2(t) \shortrightarrow x_1(t)}}{dt}(v_2(t))+R_{x_2(t) \shortrightarrow x_1(t)}(\dot v_2(t))-\dot v_1(t)=0.\end{align}

We use \eqref{two_eq} and \eqref{eq 3.2}-\eqref{eq 3.1} to obtain that
\begin{align}\label{eq dv1}
\dot v_1(0)=-(a^2+b^2)(1,0,0)^T+ \frac{n}{N}\sigma(\|x_1(0)-x_2(0)\|^2)(0,\sin\theta,0)^T
\end{align}
and
\begin{align}\label{eq dv2}
\dot v_2(0)=-(a^2+b^2)(\cos\theta,\sin\theta,0)^T+ \frac{N-n}{N}\sigma(\|x_1(0)-x_2(0)\|^2)(1-\cos^2\theta,-\sin\theta\cos\theta,0)^T.
\end{align}
By \eqref{eq dv1}-\eqref{eq dv2} and the definition of $R(\cdot,\cdot)$,
\begin{align}\label{eq 3.4}
R_{x_2(0) \shortrightarrow x_1(0)}(\dot v_2(0))-\dot v_1(0)=\left(0,~-\sigma(\|x_1(0)-x_2(0)\|^2)\sin\theta,~0\right)^T.\end{align}
Then, by elementary calculation and the definition of the rotation operator $R$, we have
\begin{align}\begin{aligned}\label{eq 3.6}
\frac{d}{dt}   R(x_2(t),x_1(t))
&=\langle v_2(t) , x_1(t)\rangle I+ \langle x_2(t) , v_1(t)\rangle I
- v_2(t) x_1(t)^T
\\
&\qquad- x_2(t) v_1(t)^T
+ v_1(t) x_2(t)^T
+ x_1(t) v_2(t)^T
\\&\qquad
 + \frac{d}{dt}\left(\frac{1-  \langle x_2(t), x_1(t)\rangle}{|x_2(t) \times x_1(t)|^2}\right) ( x_2(t) \times x_1(t) ) ( x_2(t) \times x_1(t))^T
\\&\qquad
+\left(\frac{1-  \langle x_2(t), x_1(t)\rangle}{|x_2(t) \times x_1(t)|^2}\right) ( v_2(t) \times x_1(t) ) ( x_2(t) \times x_1(t))^T\\&
\qquad
+\left(\frac{1-  \langle x_2(t), x_1(t)\rangle}{|x_2(t) \times x_1(t)|^2}\right) ( x_2(t) \times v_1(t) ) ( x_2(t) \times x_1(t))^T
\\&\qquad
+\left(\frac{1-  \langle x_2(t), x_1(t)\rangle}{|x_2(t) \times x_1(t)|^2}\right) ( x_2(t) \times x_1(t) ) ( v_2(t) \times x_1(t))^T\\&
\qquad
+\left(\frac{1-  \langle x_2(t), x_1(t)\rangle}{|x_2(t) \times x_1(t)|^2}\right) ( x_2(t) \times x_1(t) ) ( x_2(t) \times v_1(t))^T.\end{aligned}
  \end{align}
Clearly, we have
\begin{align}\begin{aligned}\label{eq 3.7}
  \frac{d}{dt}\left(\frac{1-  \langle x_2(t), x_1(t)\rangle}{|x_2(t) \times x_1(t)|^2}\right)&=\frac{-  \langle v_2(t), x_1(t)\rangle-  \langle x_2(t), v_1(t)\rangle}{|x_2(t) \times x_1(t)|^2}\\
  &\quad-\frac{2- 2 \langle x_2(t), x_1(t)\rangle}{|x_2(t) \times x_1(t)|^4}\langle v_2(t) \times x_1(t)+x_2(t) \times v_1(t),x_2(t) \times x_1(t)\rangle .
\end{aligned}\end{align}
  If we use  \eqref{eq 3.2}-\eqref{eq 3.3} and \eqref{eq 3.6}-\eqref{eq 3.7}, then we obtain
  \phantom{\eqref{eq 3.6}\eqref{eq 3.4}}
  \begin{align}\label{eq 3.10}
    \left[\frac{d}{dt}   R(x_2(t),x_1(t))\right]\cdot v_2(t)\bigg|_{t=0}=\left(0,~-2b^2\tan\frac{\theta}{2}, ~2 ab\tan\frac{\theta}{2}\right)^T.
  \end{align}
From \eqref{eq 3.5},\eqref{eq 3.4} and \eqref{eq 3.10}, it follows that
  \[(0,0,0)=\left(0,~-2\Big(b^2+\frac{\sigma_{12}}{2}+\frac{\sigma_{12}}{2}\cos \theta\Big)\tan\frac{\theta}{2},~ 2ab \tan\frac{\theta}{2}    \right),\]
  where $\sigma_{12}=\sigma(\|x_1(0)-x_2(0)\|^2)$. Therefore,  only two cases are possible:
  \[b=0,~ \sigma_{12}=0\]
  or
  \[a=0, ~b=\pm\sqrt{-\frac{\sigma_{12}}{2}(1+\cos \theta)}.\]
From this result, we complete the proof of this proposition.
\end{proof}

The following corollary of Proposition \ref{prop 3.3} gives a special property for the case of $N=2$. From this, it follows that the result in Theorem \ref{thm:2} does not hold for  $N=2$. See Remark \ref{rmk 1.2}. This proposition also will be used in the proof of the main theorem.

\begin{corollary}\label{propN2}We assume that $\psi(\cdot)$ is a nonnegative bounded function satisfying \eqref{eqn:psi} and $\sigma(\cdot)$ is defined by \eqref{sigma_def}. If $N=2$ and $\sigma_r>0$, then the largest positively invariant set $\Omega_M$ consists of the points of the following trajectories.
\begin{itemize}
\item[(1)]Their relative position is   antipodal.
\item[(2)]The two nodes move around  the same great circle at the same constant speed with $\sigma(\|x_1-x_2\|^2)=0$.
\item[(3)]Each node rotates at the speed
\[\sqrt{\frac{-\sigma(\|x_1-x_2\|^2)}{2}(1+\cos \theta)}\]
on one of two parallel circles that have the same radius  with distance $\displaystyle 2\sin\frac{\theta}{2}$,
 where $\theta$ is a constant between $0$ and $\pi$.
\end{itemize}
\end{corollary}
The third case in Corollary \ref{propN2} is unstable. We next focus on the case of $N\geq 3$. The following simple lemma plays a crucial role in the proof of the main theorem. The geometric properties of $\mathbb{S}^2$ and the rotation operator give the following rigid motions.
\begin{lemma}\label{lemma 1.5}
 Assume that $\{(x_1,v_1),(x_2,v_2),(x_3,v_3)\}\subset T\bbs^2$ are velocity aligned and their positions are not antipodal to each other,
   i.e,
  \[R_{x_j \shortrightarrow x_i}(v_j)-v_i=0,\quad \mbox{for any ~$i,j\in \{1,2,3\}$.}\]
    Then $v_1=v_2=v_3=0$ or $\{x_1,x_2,x_3\}$ are located on  a great circle.
\end{lemma}
\begin{proof}We assume that $\{x_1,x_2,x_3\}$ are not located on a great circle. Then, we can consider the triangle on  $\bbs^2$ with three points $\{x_1,x_2,x_3\}$  as vertices. Their line segment is a part of a great circle connecting  two points of $\{x_1,x_2,x_3\}$. Then the sum of the interior angles of any spherical triangle is strictly greater than $\pi$. We assume that $v_1\ne 0$ and
\[R_{x_1 \shortrightarrow x_2}(v_1)=v_2,\quad R_{x_2 \shortrightarrow x_3}(v_2)=v_3,\quad R_{x_3 \shortrightarrow x_1}(v_3)=v_1.\]
That means
\[R_{x_3 \shortrightarrow x_1}\circ R_{x_2 \shortrightarrow x_3}\circ R_{x_1 \shortrightarrow x_2}(v_1)=v_1.\]
However, the above holds only for $v_1=0$ or  the sum of the exterior angles of the spherical triangle is  $2\pi$. Since the sum of the interior angles of any spherical triangle is strictly greater than $\pi$, the sum of the exterior angles of the spherical triangle is  not $2\pi$. Therefore, we obtained the desired result.
\end{proof}

\begin{lemma}\label{lemma 3.5}
We assume that  $\psi_{ij}$  satisfies  \eqref{eqn:psi_ij}-\eqref{eqn:psi}
 and $\sigma_{ij}$ is defined by \eqref{sigma}-\eqref{sigma_def}.  Let  $M_{\E_0}\subset \Omega_{\E_0}$ be the largest positively invariant set of \eqref{main} satisfying
\begin{align}\label{eq 3.11}
\dot\E=-\sum_{i,j=1}^N\frac{\psi(\|x_i-x_j\|)}{2N^2}\| R_{x_j \shortrightarrow x_i}(v_j)-v_i\|^2=0.
\end{align}
  If $N\geq 3$ and $\sigma_r>0$, then for any particle trajectory $\{(x_i(t),v_i(t))\}_{i=1}^N$ of  \eqref{main} in $M_{\E_0}$,  there is $t_0\geq 0$ such that one of the followings holds.
\begin{itemize}
 \item[(a)] $v_i(t_0)=0$ for all $i\in \{1,\ldots,N\}$ and $\{x_i(t_0)\}_{i=1}^N$ are not deployed on one great circle,
 \item[(b)] $\{x_i(t_0)\}_{i=1}^N$ are deployed on a great circle and their velocities are aligned along the great circle.

 \end{itemize}
\end{lemma}
\begin{proof}Let
\[\psi_{ij}(t)=\psi(\|x_i(t)-x_j(t)\|).\]
We divide it into the following three cases:
\[N\geq 5, \quad N=4, \quad N=3.\]

For $N\geq 5$,  we take any trajectory  $\{(x_i(t),v_i(t))\}_{i=1}^N$ of \eqref{main} in $M_{\E_0}$. We assume that there is $t_0\in \bbr$ such that  $(x_i(t_0),v_i(t_0))_{i=1}^N$ are not deployed on a great circle. Since  $N\geq 5$, there are  three points  $\{x_i(t_0),x_j(t_0),x_k(t_0)\}$  such that any two points of them are not antipodal. This implies that
\begin{align}\label{eq 3.12}
\psi_{ij}(t_0),\psi_{jk}(t_0), \psi_{ki}(t_0)\ne 0.
\end{align}
By \eqref{eq 3.11}-\eqref{eq 3.12} and Lemma \ref{lemma 1.5},
\begin{align}\label{eq 3.135}
  v_i(t_0)=v_j(t_0)=v_k(t_0)=0
\end{align}
  or $\{x_i(t_0),x_j(t_0),x_k(t_0)\}$ are deployed on a great circle.

If we assume that \eqref{eq 3.135} does not hold, then $\{x_i(t_0),x_j(t_0),x_k(t_0)\}$ are deployed on a great circle. For this case, we claim that   $\{v_i(t_0), v_j(t_0), v_k(t_0)\}$ are aligned along their great circle.
\begin{proof}[The proof of the claim:]
If one of the velocities is a zero vector, then all velocities are zero since they are not antipodal to each other and velocities are aligned. Then we are done.

Next, we consider the case that all velocities are nonzero.   We assume that the velocities are not aligned along the great circle.  Then for sufficiently small $\epsilon>0$,  $\{x_i(t_0+\epsilon),x_j(t_0+\epsilon),x_k(t_0+\epsilon)\}$   does not located on the great circle but $v_i(t_0+\epsilon),v_j(t_0+\epsilon),v_k(t_0+\epsilon)$ are not zero vector by the continuity of the solution. Therefore, any two vectors of $\{x_i(t_0+\epsilon),x_j(t_0+\epsilon),x_k(t_0+\epsilon)\}$ are not antipodal. Therefore,  it does not satisfy \[v_i(t_0+\epsilon)=v_j(t_0+\epsilon)=v_k(t_0+\epsilon)=0\]
  nor $\{x_i(t_0+\epsilon),x_j(t_0+\epsilon),x_k(t_0+\epsilon)\}$ are deployed on a great circle. Thus, it is not invariant by Lemma \ref{lemma 1.5}. Therefore, if three points are deployed on a great circle, then their velocities are aligned along the great circle.
\end{proof}

   Thus, we have proved that one of the following statements holds.
 \begin{itemize}
 \item[(a')] $v_i(t_0)=v_j(t_0)=v_k(t_0)=0$ and $\{x_i(t_0),x_j(t_0),x_k(t_0)\}$ are not deployed on a great circle,
 \item[(b')]$\{x_i(t_0),x_j(t_0),x_k(t_0)\}$ are deployed on a great circle and their velocities are aligned along the great circle.

 \end{itemize}

Next, we consider the rest of the agents $\{(x_l,v_l)\}_{1\leq l\leq N, l\ne i, j,k}$ for both cases $(a')$ and $(b')$ in the above. For (a'), we fixed the $l$th  agent with  position $x_l(t_0)$, $1\leq l\leq N$, $l\ne i, j,k$. Then, at least one of $\{x_i(t_0),x_j(t_0),x_k(t_0)\}$ is not antipodal with $x_l(t_0)$. Thus, by \eqref{eq 3.11}, $v_l(t_0)=0$ for all $l\in \{1,\ldots, N\}$. Therefore, in this case, all velocities are the zero vector. For (b'), we also fix the $l$th particle with $(x_l(t_0),v_l(t_0))$. If $x_l(t_0)$ is located on the great circle containing $\{x_i(t_0),x_j(t_0),x_k(t_0)\}$, then  $v_l(t_0)$ is aligned along the great circle by the same argument in the above.

If there is  an index $l\in \{1,\ldots,N\}-\{i,j,k\}$ such that $x_l(t_0)$ is not located on the great circle  containing $\{x_i(t_0),x_j(t_0),x_k(t_0)\}$, then we can find  three points in $\{x_i(t_0),x_j(t_0),x_k(t_0), x_l(t_0)\}$ and the three points are not antipodal to each other and they are not located on any great circle. Thus, by Lemma \ref{lemma 1.5}, their velocities are all the zero vectors. If there is $t_0\in \bbr$ such that  $\{(x_i(t_0),v_i(t_0))\}_{i=1}^N$ are not deployed on a great circle, then we obtain the result for $N\geq 5$. Otherwise, for all $t\geq 0$, $\{(x_i(t),v_i(t))\}_{i=1}^N\in M_{\E_0}$ are located on the great circle. Similar to the previous case, we can verify that their velocities are aligned along the great circle.

For $N=4$, there is $t_0\in \bbr$ such that  $\{(x_i(t_0),v_i(t_0))\}_{i=1}^N$ are not deployed on a great circle or not. We first assume that there is $t_0\in \bbr$ such that  $\{(x_i(t_0),v_i(t_0))\}_{i=1}^N$ are not deployed on a great circle. Then we can find three position in  $\{(x_i(t_0),v_i(t_0))\}_{i=1}^N$, say
\[S=\{(x_i(t_0),v_i(t_0)),(x_j(t_0),v_j(t_0)),(x_k(t_0),v_k(t_0))\},\]
 such that $S$ are not deployed on a great circle and any two agents in $S$  are not antipodal. Then by the same argument in the case of $N\geq 5$, we obtain the result.

 The remaining case is that four points $\{(x_i(t_0),v_i(t_0))\}_{i=1}^N$ are located on a great circle for all $t\geq 0$. Then we can divide it into two subcases: they form  two pairs of antipodal points or not. If they do not form two pairs of antipodal points, then  there are  three points  $\{x_i(t_0),x_j(t_0),x_k(t_0)\}$  such that any two points of them are not antipodal. Thus, by the previous argument, this lemma holds. Next, we consider  two pairs of antipodal points. Take $(x_1(t_0),v_1(t_0))$ in $\{(x_i(t_0),v_i(t_0))\}_{i=1}^4$. Without loss of generality, we let $x_3(t_0)$ be the antipodal points of $x_1(t_0)$. Thus, $x_2(t_0)$ and $x_4(t_0)$ are antipodal each other. Then we have
\begin{align}\label{eq 3.13}
  R_{x_1 \shortrightarrow x_2}(v_1)=v_2,\quad R_{x_2 \shortrightarrow x_3}(v_2)=v_3,\quad R_{x_3 \shortrightarrow x_4}(v_3)=v_4,\quad \quad R_{x_4 \shortrightarrow x_1}(v_4)=v_1.
\end{align}

Similar to  the case of $N\geq 5$, we assume that the velocities are not aligned along the great circle.  Then there is a sufficiently small $\epsilon>0$ such that  $\{x_1(t_0+\epsilon),x_2(t_0+\epsilon),x_3(t_0+\epsilon), x_4(t_0+\epsilon)\}$  does not located on a great circle and $\{v_1(t_0+\epsilon),v_2(t_0+\epsilon),v_3(t_0+\epsilon),v_4(t_0+\epsilon)\}$ are not the zero vector by the continuity of the solution. Moreover, by \eqref{eq 3.13}, the four points are located on a hemisphere. Hence,  the four points are not located on the same great circle and any two vectors of $\{x_1(t_0+\epsilon), x_2(t_0+\epsilon), x_3(t_0+\epsilon), x_4(t_0+\epsilon)\}$ are not antipodal. Since their velocities are nonzero, it is not invariant. Using the same method in  the case of $N\geq 5$, we can prove that the velocities are aligned along the great circle or all velocities are zero vectors.

Finally, we consider $N=3$. If they are not on a great circle, then there is no pair of antipodal points. Therefore, all velocities are zero by Lemma \ref{lemma 1.5}. Thus, it suffices to consider that they are located on a great circle. Then by a similar argument in the case of $N=4$, we can also prove that the velocities are aligned along the great circle or all velocities are zero vectors.

\end{proof}

\begin{proposition}\label{prop 1.6}We assume that  $\psi_{ij}$  satisfies  \eqref{eqn:psi_ij}-\eqref{eqn:psi}
 and $\sigma_{ij}$ is defined by \eqref{sigma}-\eqref{sigma_def}.  Let  $M_{\E_0}\subset \Omega_{\E_0}$  be the largest positively invariant set of \eqref{main} satisfying
\begin{align*}
\sum_{i,j=1}^N\frac{\psi(\|x_i-x_j\|)}{2N^2}\| R_{x_j \shortrightarrow x_i}(v_j)-v_i\|^2=0.
\end{align*}
 If $N\geq 3$ and $\sigma_r>0$, then $M_{\E_0}$ consists of the points of the following trajectories:
\begin{itemize}
\item[(1)] a stationary state, i.e.,  $v_i(t)=0$ for all $i\in \{1,\ldots,N\}$,
\item[(2)] the all nodes moving around  the same great circle at the same constant speed.
\end{itemize}
Moreover, for both cases of (1) and (2), the following holds for $i\in \{1,\ldots,N\}$.
\begin{align}\label{eq 3.15}
  \sum_{j=1,j\ne i}^N \frac{\sigma(\|x_i-x_j\|^2) }{N}(\|x_i\|^2x_j - \langle x_i,x_j \rangle  x_i)=0.
\end{align}

\end{proposition}
\begin{proof}
By Lemma \ref{lemma 3.5}, if $N\geq 3$, then for any solution $\{(x_i(t),v_i(t))\}_{i=1}^N \in M_{\E_0}$, there is $t_0\geq 0$ such that
 \begin{itemize}
   \item $v_i(t_0)=0$ for all $i\in \{1,\ldots,N\}$ and $\{x_i(t_0)\}_{i=1}^N$ are not deployed on a great circle~~~~~
   \[\mbox{or}\]
   \item $\{x_i(t_0)\}_{i=1}^N$ are deployed on a great circle and their velocities are aligned along the circle.
 \end{itemize}
\noindent$\circ$ Case 1:  all agents $\{x_i(t_0)\}_{i=1}^N$ are not deployed on  one great circle.\\

By the above, $v_i(t_0)=0$ for all $i\in \{1,\ldots,N\}$.
In this case, we claim that for all $t\geq 0$ and $i\in \{1,\ldots,N\}$, $v_i(t)=0$.
Assume not, i.e., there is $t_1\geq 0$ and an index $i_1\in \{1,\ldots,N\}$ such that
\[v_{i_1}(t_1)\ne 0.\]
Without loss of generality, we may assume that $t_1>t_0$. Then there is the maximum interval $[t_0,T]$ such that for all $1\leq i\leq N$,
\[v_i(t)=0, \quad t\in  [t_0,T]\]
and  there is $i_1'\in\{1,\ldots, N\}$ such that
\[v_{i_1'}(t)\neq 0, \quad t\in (T,T'),\]
 for some $T'>0$. Therefore, at $t=T$, $\{x_i(t)\}_{i=1}^N$ are not deployed on a great circle.
 By the continuity of the solution, we can choose small $\epsilon>0$ such that
 \begin{align}\label{eq 3.14}
 v_{i_1'}(T+\epsilon)\neq 0
 \end{align}
 and $\{x_i(T+\epsilon)\}_{i=1}^N$ are not deployed on an one great circle. Thus, there are three points in $\{x_i(T+\epsilon)\}_{i=1}^N$ such that any two points in the three points are not antipodal.  Since $\{(x_i(t),v_i(t))\}_{i=1}^N \in M_{\E_0}$, we have
 \[\sum_{i,j=1}^N\frac{\psi(\|x_i(t)-x_j(t)\|)}{2N^2}\| R_{x_j(t) \shortrightarrow x_i(t)}(v_j(t))-v_i(t)\|^2=0\quad \mbox{ at $t=T+\epsilon$.}\]
  This implies that $\| R_{x_j(T+\epsilon) \shortrightarrow x_i(T+\epsilon) }(v_j(T+\epsilon) )-v_i(T+\epsilon) \|^2=0$ for these three agents. Thus, by Lemma \ref{lemma 1.5}, their velocities at $t=T+\epsilon$ are all zero vectors. By similar argument in Lemma \ref{lemma 3.5}, other agents also have zero velocity. This is a contradiction to  \eqref{eq 3.14}. Since $\{(x_i(t),v_i(t))\}_{i=1}^N \in M_{\E_0}$ is the solution to \eqref{main},  \eqref{eq 3.15} holds for all  $t\geq 0$ and $i\in \{1,\ldots,N\}$.\\

\noindent $\circ$ Case 2:  all $\{x_i(t_0)\}_{i=1}^N$ are deployed on a great circle and their velocities are aligned along the  circle.\\

Their trajectories  also keep the same great circle. Therefore, after changing the coordinates,  we can represent the solution by
\[x_i=(\cos\theta_i,\sin \theta_i,0),\quad i=1,\ldots,N.\]
Therefore, we have
\[v_i=\dot\theta_i(-\sin\theta_i,\cos \theta_i,0),\quad i=1,\ldots,N\]
and
\[\dot v_i=\ddot\theta_i(-\sin\theta_i,\cos \theta_i,0)-(\dot\theta_i)^2(\cos\theta_i,\sin \theta_i,0),\quad i=1,\ldots,N.\]
Moreover, by the velocity alignment,  for any $i,j\in\{1,\ldots,N\}$,
\begin{align}\label{thetadot}
\dot \theta_i=\dot \theta_j.
\end{align}

Since $\psi_{ij}\big(R_{x_j \shortrightarrow x_i}(v_j)-v_i\big)=0$ for any $i,j\in\{1,\ldots,N\}$, the second equation in the main system gives the following equality.
\[\ddot\theta_i(-\sin\theta_i,\cos \theta_i)=\sum_{j=1,j\ne i}^N \frac{\sigma_{ij} }{N}\left[(\cos\theta_j,\sin \theta_j) - \cos (\theta_i-\theta_j)(\cos\theta_i,\sin \theta_i)\right].\]
 This yields that
 \begin{align}\label{eq 3.16}\ddot\theta_i=\sum_{j=1,j\ne i}^N \frac{\sigma_{ij} }{N}\sin (\theta_j-\theta_i).\end{align}

 The right hand side of the above is a constant for $t>0$ by \eqref{thetadot}. Therefore, if $\ddot\theta_i \ne 0$, then $\dot\theta_i$ goes to $\infty$ or $-\infty$. However, it is impossible since $\E$ is bounded. Thus, we have $\ddot\theta_i = 0$ and it is a constant rigid motion. Finally, \eqref{eq 3.16} and $\ddot\theta_i=0$ implies \eqref{eq 3.15}.
\end{proof}

\begin{lemma}\cite{L-S}\label{lemma 3.6} Let $\{(x_i(t),v_i(t))\}_{i=1}^N $ be a solution to \eqref{main} and $x_i(t)\in \bbs^2$ for all  $i\in \{1,\ldots,N\}$.
We assume that $\sigma_{ij}$ is defined by \eqref{sigma} and the following holds.
\begin{align}\label{eq 3.165}
  \sum_{j=1,j\ne i}^N \frac{\sigma(\|x_i(t)-x_j(t)\|^2) }{N}(\|x_i(t)\|^2x_j(t) - \langle x_i(t),x_j(t) \rangle  x_i(t))=0,\quad \mbox{for all  $i\in \{1,\ldots,N\}$}.
\end{align}
Then
\begin{align*}
\Big(N\sigma_a-\frac{(N-1)\sigma_r}{2}\Big)\|\bar{x}(t)\|^2=\sigma_a\sum_{i=1}^N \langle x_i(t),\bar{x}(t) \rangle^2,
\end{align*}
where $\bar{x}$ is the centroid of the agents.
\end{lemma}

\begin{proof}

By \eqref{eq 3.165}, $x_i\in \bbs^2$ and  $\sigma_{ij}=\sigma_{ji}$, we have
\begin{align*}
0=\sum_{i,j=1,j\ne i}^N \sigma_{ij} (x_i - \langle x_i,x_j \rangle  x_i)=\sum_{i,j=1,j\ne i}^N \left(\sigma_a-\frac{\sigma_r}{\|x_i-x_j\|^2}\right)(x_i - \langle x_i,x_j \rangle  x_i).
\end{align*}
Since $\|x_i-x_j\|^2=2-2\langle x_i,x_j\rangle$,
\begin{align*}
0=\frac{1}{2}\sum_{i,j=1,j\ne i}^N \left(\sigma_a\|x_i-x_j\|^2-\sigma_r\right)x_i
=\sum_{i=1}^N \left(N\sigma_a-N\sigma_a\langle x_i,\bar{x} \rangle-\frac{(N-1)\sigma_r}{2}\right)x_i.
\end{align*}
Therefore, we have
\begin{align*}
0=\sum_{i=1}^N \left(\sigma_a-\sigma_a\langle x_i,\bar{x} \rangle-\frac{N-1}{N}\frac{\sigma_r}{2}\right)\langle x_i,\bar{x}\rangle=
\Big(N\sigma_a-\frac{(N-1)\sigma_r}{2}\Big)\|\bar{x}\|^2-\sigma_a\sum_{i=1}^N \langle x_i,\bar{x} \rangle^2.
\end{align*}

\end{proof}

\begin{proof}[\bf Proof of Theorem \ref{thm:2}] We will use Proposition \ref{prop 3.2} and the classification of the largest positively invariant set in Proposition \ref{prop 1.6}.
Let $M_{\E_0}$ be the largest positively invariant set of \eqref{main} satisfying
\begin{align*}
\sum_{i,j=1}^N\frac{\psi(\|x_i-x_j\|)}{2N^2}\| R_{x_j \shortrightarrow x_i}(v_j)-v_i\|^2=0\quad \mbox{and}\quad  \E(Z) \leq \E_0\quad \mbox{for}\quad Z\in M_{\E_0}.
\end{align*}
\newline

 For $(i)$, we assume that $\sigma_a>0$, $\sigma_r =0$ and $\E_0=\E(Z(0))<\E_C^0$. We claim that for any $\{(x_i,v_i)\}_{i=1}^N\in M_{\E_0}$,
  \[x_i=x_j,\quad \mbox{for any}~i,j\in \{1,\ldots,N\}.\] We take any trajectory $\{(x_i(t),v_i(t))\}_{i=1}^N\in M_{\E_0}$ and fix $t\geq 0$.  Then the following three cases are possible.
  \begin{enumerate}
    \item[$\circ$]  all vectors in $\{x_1(t),\ldots,x_N(t)\}$ are the same,
    \item[$\circ$]  there are only two different vectors in $\{x_1(t),\ldots,x_N(t)\}$,
    \item[$\circ$]  at least, there are  three different vectors in $\{x_1(t),\ldots,x_N(t)\}$.
  \end{enumerate}

 We assume that there are only two different vectors in $\{x_1(t),\ldots,x_N(t)\}$.   For simplicity, we may assume that $x_1(t)\ne x_2(t)$. Then  we can easily check that $\{(x_1(t),v_1(t)), (x_2(t),v_2(t))\}$  is the solution to  \eqref{two_eq}. Thus, we can apply Proposition \ref{prop 3.3} to obtain that $x_1(t)$ and $x_2(t)$ are antipodal since the case (2) in Proposition \ref{prop 3.3} does not occurs by
\[\sigma(\|x_i-x_j\|^2)=\sigma_a> 0.\]
This implies that any two vectors in $\{x_1(t),\ldots,x_N(t)\}$ are the same or antipodal. However, all configurations in this case violate the assumption: $\E(Z) \leq \E_0$.

Assume that at least, there are  three different vectors in $\{x_1,\ldots,x_N\}$. By Proposition \ref{prop 1.6},
 they are a stationary state or constant speed  motions on a great circle and satisfying
 \[0=\sum_{j=1,j\ne i}^N \frac{\sigma_{ij} }{N}(\|x_i\|^2x_j - \langle x_i,x_j \rangle  x_i).\]
 By Lemma \ref{lemma 3.6},
 \begin{align*}
\|\bar{x}\|^2=\frac{1}{N}\sum_{i=1}^N \langle x_i,\bar{x} \rangle^2.
\end{align*}
This means that $x_i=x_j$   for all indices $i,j\in \{1,\ldots,N\}$ and  this is a contradiction.

Therefore, all positions in $M_{\E_0}$ satisfy that $x_i=x_j$  for any $i,j\in\{1,\ldots,N\}$ and this implies that  the solution  has the asymptotic rendezvous property by Theorem \ref{LITHM}.\\

Next we consider $(ii)$. Assume that  $\displaystyle\frac{2N}{N-1}\sigma_a>\sigma_r>0$ and $\E(Z(0))<\E_{C}^1$. We take any $\{(x_i(t),v_i(t))\}_{i=1}^N\in M_{\E_0}$ and fix time $t\geq 0$.  By Proposition \ref{prop 1.6}, $\{(x_i(t),v_i(t))\}_{i=1}^N\in M_{\E_0}$ is a stationary state or constant speed  motions on a great circle and satisfies that  for any $i=1,\ldots,N$,
\[0=\sum_{j=1,j\ne i}^N \frac{\sigma_{ij} }{N}(\|x_i\|^2x_j - \langle x_i,x_j \rangle  x_i).\]
Since  $\E(Z(0))<\E_{C}^1$ and $\displaystyle\frac{d\E(Z(t))}{dt} \leq 0$, we have $\bar{x}\neq 0$.  By  Lemma \ref{lemma 3.6},
\[0=\Big(N\sigma_a-\frac{(N-1)\sigma_r}{2}\Big)\|\bar{x}\|^2-\sigma_a\sum_{i=1}^N \langle x_i,\bar{x} \rangle^2.\]
This implies that
\begin{align*}
(1-\frac{N-1}{N}\frac{\sigma_r}{2\sigma_a})\|\bar{x}\|^2=\frac{1}{N}\sum_{i=1}^N \langle x_i,\bar{x} \rangle^2.
\end{align*}
Note that
\[\rho(t)+\frac{1}{\|\bar{x}\|^2}\sum_{i=1}^N \langle x_i,\bar{x} \rangle^2=N.\]
Therefore,  $\displaystyle\rho(t)=\frac{(N-1)\sigma_r}{2\sigma_a}$ for all $t\geq 0$ and we conclude that the ensemble  has an asymptotic formation configuration by Theorem \ref{LITHM}.\\

Finally, we deal with $(iii)$. Let  $\E_0$ be the initial energy of a fixed solution $\{(x_i,v_i)\}_{i=1}^N$. We assume that  $\displaystyle\sigma_r \geq \frac{2N}{N-1}\sigma_a> 0$ or  $\displaystyle\sigma_r>0=\sigma_a$. Similar to the previous cases, if $N\geq 3$, then Proposition \ref{prop 1.6} and Lemma \ref{lemma 3.6} imply that
\[\bar{x}=0,\quad \mbox{for any}~\{(x_i,v_i)\}_{i=1}^N\in M_{\E_0}.\]
Thus, the ensemble  has an  asymptotic uniform deployment by Theorem \ref{LITHM}.

\end{proof}

\begin{remark}
By  a symmetric formation, for $N=3$, we can construct an example such that $\rho(t)$ converges to $N$ when  $\displaystyle \sigma_r < \frac{2N}{N-1}\sigma_a$. Therefore, the above result is almost optimal.
\end{remark}

\section{Simulation results}\label{sec4}\setcounter{equation}{0}
In this section, by adding a boost term to \eqref{main}, we numerically implement the nonzero-speed formation flight of the solution to the main system. We verify the results in Theorem \ref{thm:2} through numerical simulations. In \cite{L-S}, cooperative control laws were able to achieve various steady-state patterns by interacting with a damping term. However,  it is necessary to remove the damping term and add a boost term  to  keep the formation and  a non-zero speed of agents simultaneously. In this case, the desired formation was not constructed by only cooperative control (Figure \ref{fig0}). On the other hand, \eqref{main}  derives robust non-stationary formations even when the boost term is added due to interaction with the cooperative control law and the flocking term for the spherical surface. We note that  each agent maintains a nonzero  constant speed (Figure \ref{fig9}). In particular, it can be seen that the patterns of the flight formation are very similar to the stationary patterns without the boost term. See Figure \ref{fig6} for the comparison between maximal diameters of the ensemble in \eqref{main} and the boosted ensemble in \eqref{eqb}.
\begin{figure}[!ht]
\centering
\begin{minipage}{0.243\textwidth}
\centering
\includegraphics[width=\textwidth]{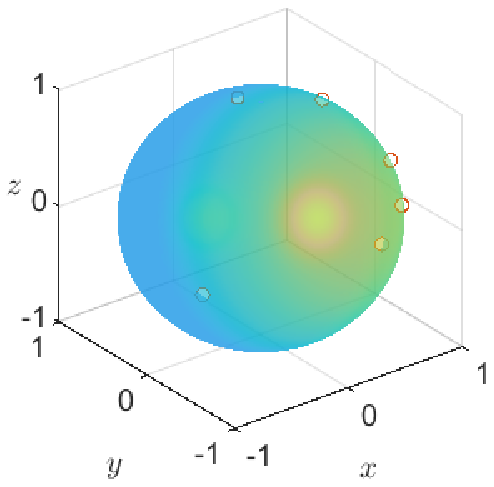}\\
(a) $t=0$
\end{minipage}
\begin{minipage}{0.243\textwidth}
\centering
\includegraphics[width=\textwidth]{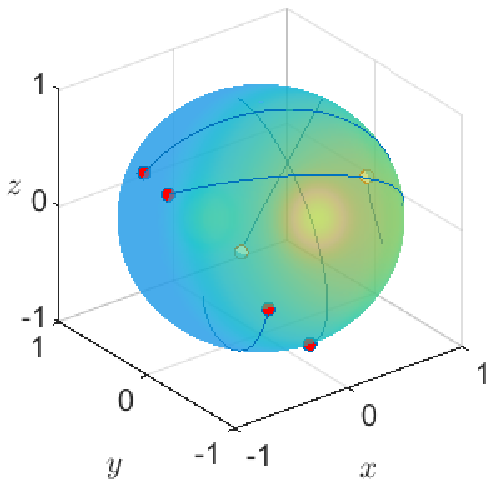}\\
(b) $t=1$
\end{minipage}
\begin{minipage}{0.243\textwidth}
\centering
\includegraphics[width=\textwidth]{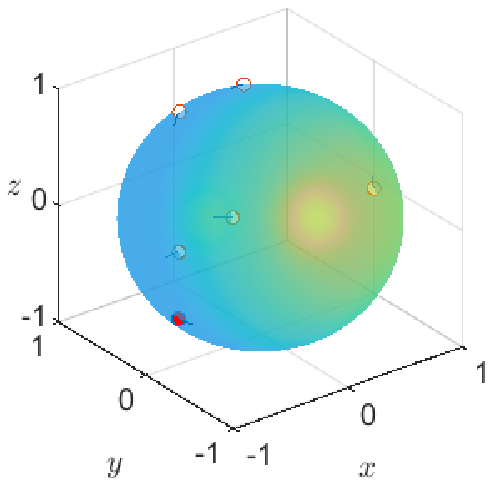}\\
(c) $t=10$
\end{minipage}
\begin{minipage}{0.243\textwidth}
\centering
\includegraphics[width=\textwidth]{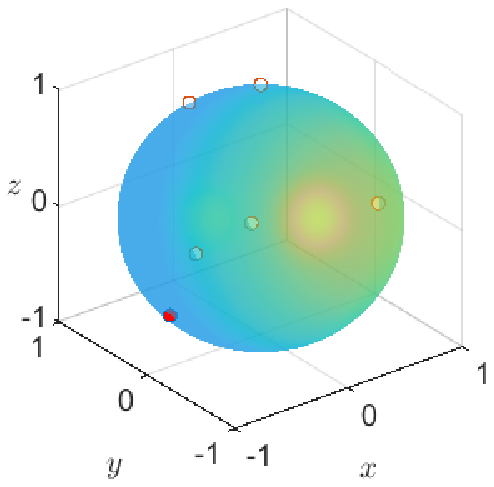}\\
(d) $t=50$
\end{minipage}
\caption{Time evolution when $\psi(x)=1$ and $\sigma_{ij}=0$}
\label{fig1}
\end{figure}

\begin{figure}[!ht]
\centering
\begin{minipage}{0.243\textwidth}
\centering
\includegraphics[width=\textwidth]{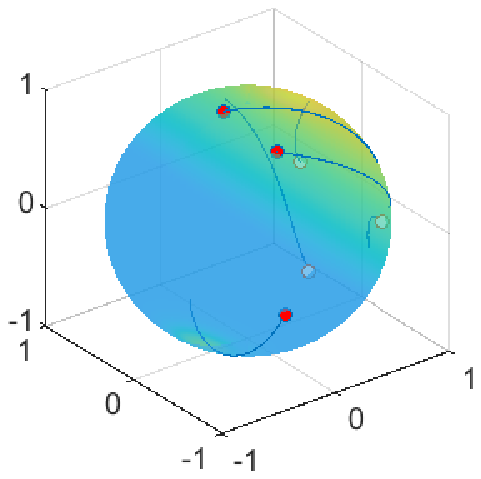}\\
(a) $t=1$
\end{minipage}
\begin{minipage}{0.243\textwidth}
\centering
\includegraphics[width=\textwidth]{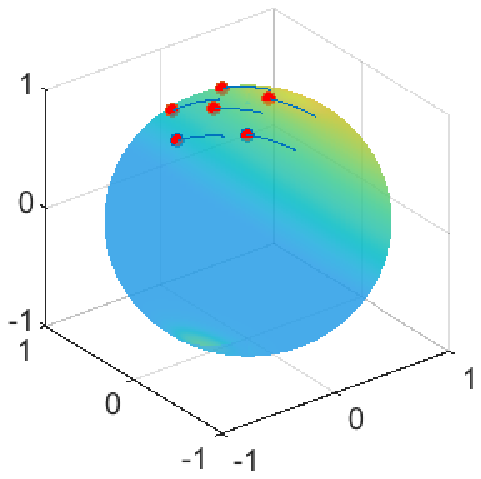}\\
(b) $t=10$
\end{minipage}
\begin{minipage}{0.243\textwidth}
\centering
\includegraphics[width=\textwidth]{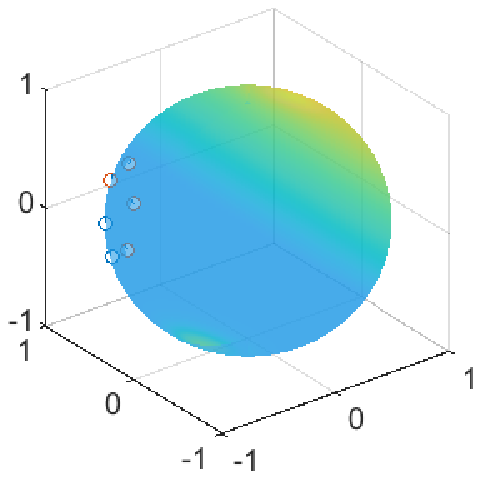}\\
(c) $t=50$
\end{minipage}
\begin{minipage}{0.243\textwidth}
\centering
\includegraphics[width=\textwidth]{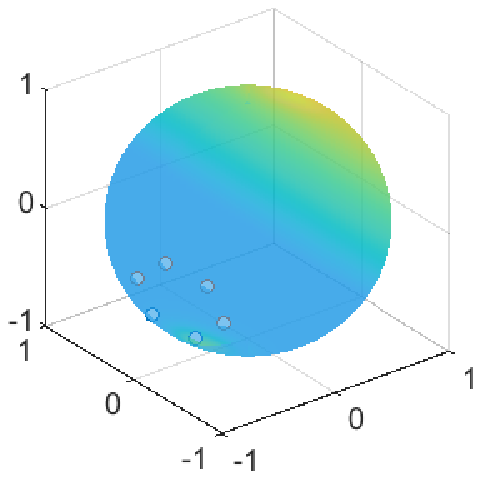}\\
(d) $t=100$
\end{minipage}
\caption{Time evolution when $\psi(x)=2(\exp(2-x)-1)$, $\sigma_a=1$ and $\sigma_r=0.5$}
\label{fig2}
\end{figure}

\begin{figure}[!ht]
\centering
\begin{minipage}{0.243\textwidth}
\centering
\includegraphics[width=0.8\textwidth]{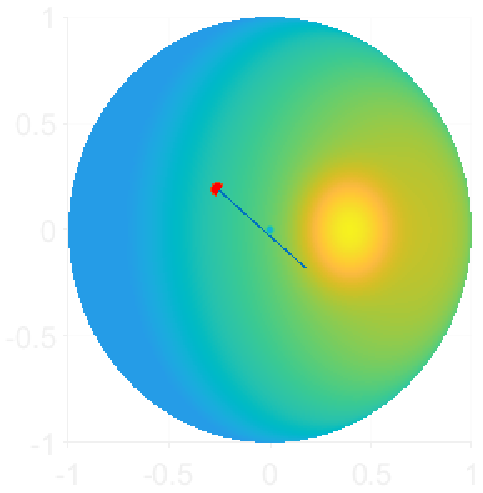}\\
(a) $\sigma_r=0$
\end{minipage}
\begin{minipage}{0.243\textwidth}
\centering
\includegraphics[width=0.8\textwidth]{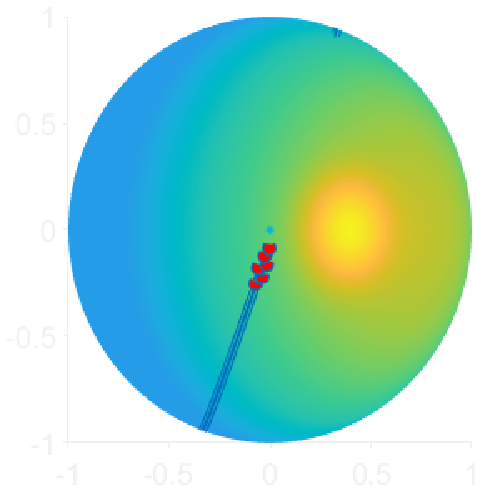}\\
(b) $\sigma_r=0.01$
\end{minipage}
\begin{minipage}{0.243\textwidth}
\centering
\includegraphics[width=0.8\textwidth]{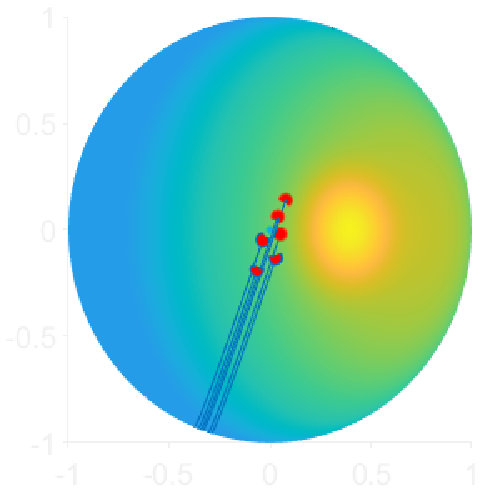}\\
(c) $\sigma_r=0.04$
\end{minipage}
\begin{minipage}{0.243\textwidth}
\centering
\includegraphics[width=0.8\textwidth]{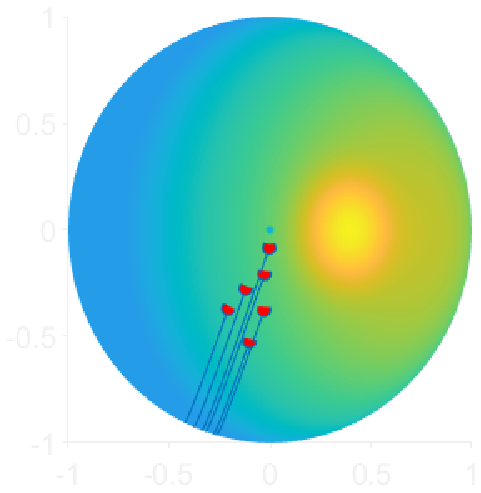}\\
(d) $\sigma_r=0.08$
\end{minipage}
\vfill
\begin{minipage}{0.243\textwidth}
\centering
\includegraphics[width=0.8\textwidth]{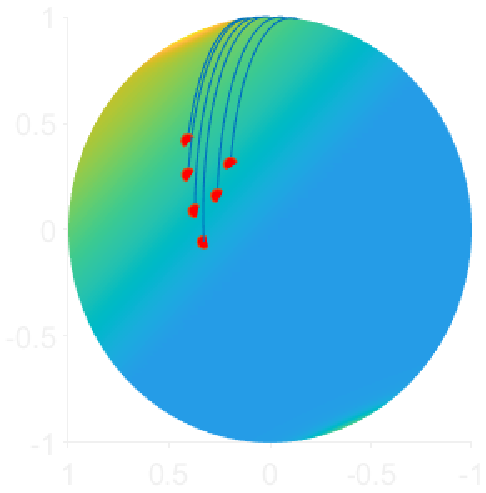}\\
(e) $\sigma_r=0.1$
\end{minipage}
\begin{minipage}{0.243\textwidth}
\centering
\includegraphics[width=0.8\textwidth]{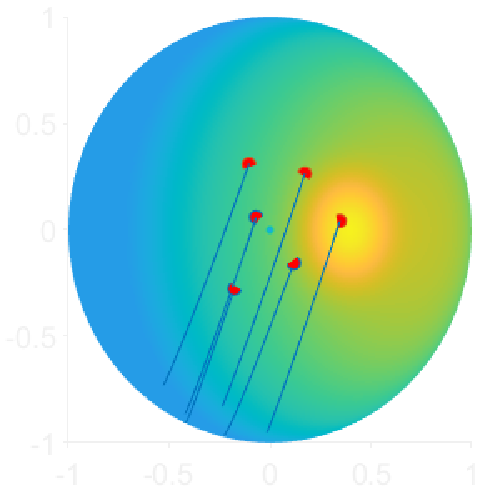}\\
(f) $\sigma_r=0.2$
\end{minipage}
\begin{minipage}{0.243\textwidth}
\centering
\includegraphics[width=0.8\textwidth]{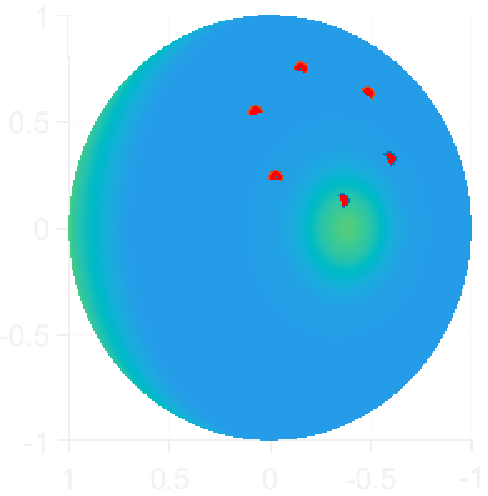}\\
(g) $\sigma_r=0.3$
\end{minipage}
\begin{minipage}{0.243\textwidth}
\centering
\includegraphics[width=0.8\textwidth]{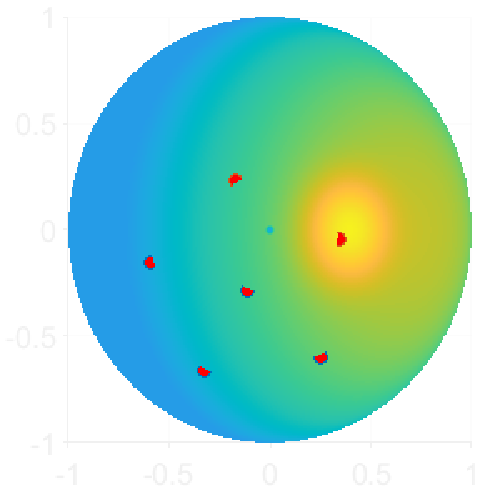}\\
(h) $\sigma_r=0.5$
\end{minipage}
\caption{Patterns of formation when $\sigma_a=1$ and $T=400$}
\label{fig3}
\end{figure}
 For the numeric simulations, we consider the flocking formation on  $\mathbb{S}^2$ with $N=6$ agents.
The initial position and velocity are randomly chosen in $(x,v)\in T\mathbb{S}^2\cap\left([-1,1]^3\times [-3,3]^3\right)$ as follows:
\begin{align*}
&x_1(0)= (-0.1192,\phantom{-}0.5108 ,-0.8514), && x_2(0)=( \phantom{-}0.8547 , -0.3671, \phantom{-}0.3671),\\
& x_3(0)=(\phantom{-}0.7076 ,  \phantom{-}0.2235,   \phantom{-}0.6704  ),&&
x_4(0)=(\phantom{-} 0.3600,\phantom{-}  0.7364, \phantom{-} 0.5728  ),\\
&  x_5(0)=( \phantom{-} 0.8977, -0.4406, \phantom{-}0.0000 ),& &x_6(0)=( \phantom{-} 0.8754,  -0.2398,  -0.4197),
\end{align*}
and
\begin{align*}
&v_1(0)=( -1.1540, -1.7264,  -0.8743), &&v_2(0)=( -1.3068, \phantom{-}0.0568,   \phantom{-}3.0000 ),\\ &v_3(0)=(  \phantom{-}0.9331, \phantom{-} 1.5789,  -1.5113),
&&v_4(0)=( \phantom{-}2.7989, \phantom{-} 0.7976,   -2.7848   ), \\&v_5(0)=( -1.0254,  -2.0892,   \phantom{-} 2.5773 ),&& v_6(0)=( \phantom{-}0.4591,  \phantom{-} 0.8375,  \phantom{-}0.4789  ).
\end{align*}
Here, all initial data satisfy the admissible conditions:
\[\|x_i(0)\|=1\quad \mbox{ and }\quad\langle x_i(0),v_i(0)\rangle=0.\]

The energy can be totally dissipated when $\sigma=0$ and hence the agents are no longer moving after some time period. See Figure \ref{fig1} (c). To visually express such a phenomenon, in this section, we use a red point for the position $x_i(t)$ of $i$th agent at $t=t_0$ and the blue line in all figures that represents the trajectory of agents for the time interval  $[t_0-3, t_0]$. Thus, a red point without a blue line in some figure, for example, Figure \ref{fig1}(d), means the corresponding agent does not move in the time interval $[t-3, t]$.

 For the cooperative control law,  we first consider the following case.
\[\sigma(x)=\sigma_a-\frac{\sigma_r}{x}.\]
In this case, as seen in Section \ref{sec3}, if the attractive force parameter  $\sigma_a$ is less than the  repulsive force parameter  $\sigma_r>0$, then their configuration is evenly distributed on the entire spherical surface. Therefore, we assume $\psi(x)=2(\exp(2-x)-1)$ and take $\sigma_r<\sigma_a$.
With this decreasing function $\psi$ and $\sigma_a=1$, the pattern  formations of \eqref{main} are given in Figure \ref{fig3} for several values of $\sigma_r$.
In Figure \ref{fig3}, if $\sigma_r$ is sufficiently small, that is, the effect of repulsion is small, then agents are clustered together.
As $\sigma_r$ increases, the force to repel each other becomes stronger.
Therefore, the diameter is getting larger. Also, $\rho(t)$ converges $(N-1)\sigma_r/2\sigma_a$ as we proved in Theorem \ref{thm:2}.
See Figure \ref{fig4} and Figure \ref{fig5} for numerical simulations. We note that for all cases except for that their configuration is on a great circle, it keeps the energy of the system in \eqref{main}  being dissipated. This leads the agents to no longer move after a period of time  as we can see in Figure \ref{fig2}.

\begin{figure}[!ht]
\centering
\begin{minipage}{0.3\textwidth}
\centering
\includegraphics[width=\textwidth]{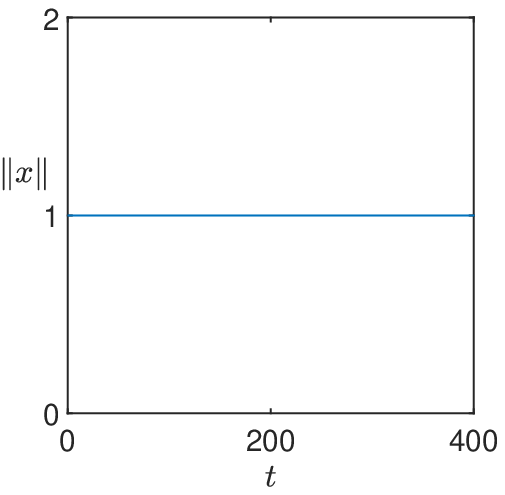}\\
(a) Norm of position $\|x_1\|$
\end{minipage}
\begin{minipage}{0.3\textwidth}
\centering
\includegraphics[width=\textwidth]{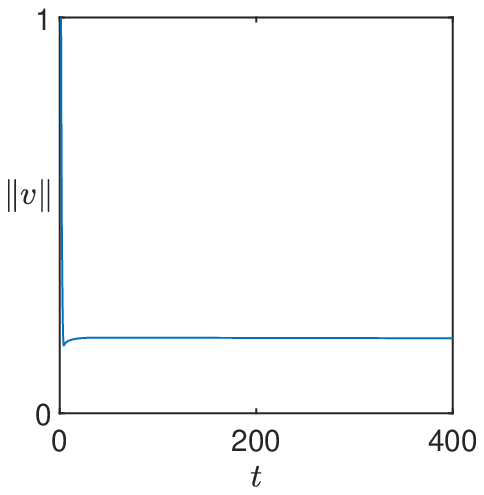}\\
(b) Norm of velocity $\|v_1\|$
\end{minipage}
\begin{minipage}{0.3\textwidth}
\centering
\includegraphics[width=\textwidth]{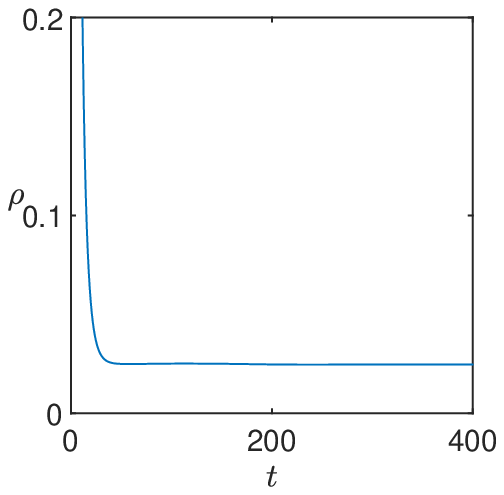}\\
(c) $\rho(t)$
\end{minipage}
\caption{The norm of position, velocity of the first agent and $\rho(t)$ when $\sigma_r=0.01$ and $\sigma_a=1$ }
\label{fig4}
\end{figure}

\begin{figure}[!ht]
\centering
\begin{minipage}{0.3\textwidth}
\centering
\includegraphics[width=\textwidth]{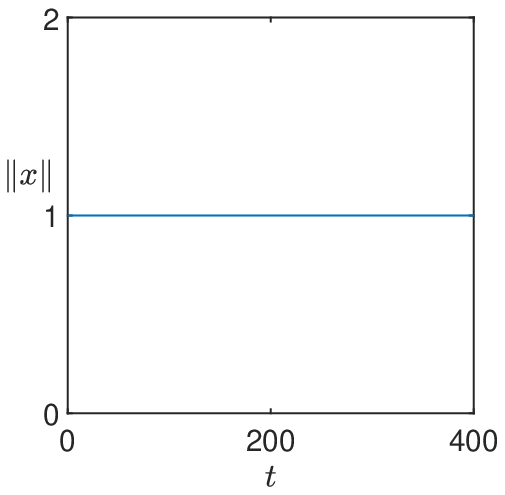}\\
(a) Norm of position $\|x_1\|$
\end{minipage}
\begin{minipage}{0.3\textwidth}
\centering
\includegraphics[width=\textwidth]{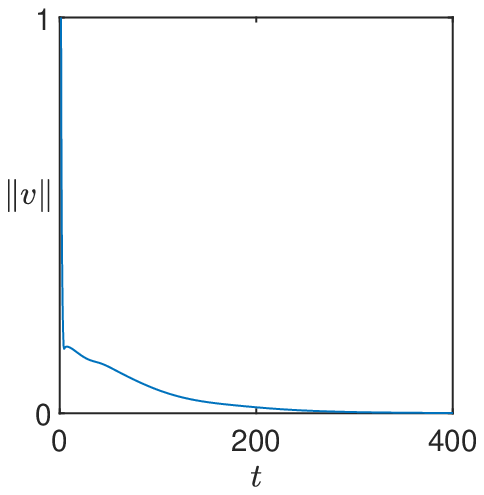}\\
(b) Norm of velocity $\|v_1\|$
\end{minipage}
\begin{minipage}{0.3\textwidth}
\centering
\includegraphics[width=\textwidth]{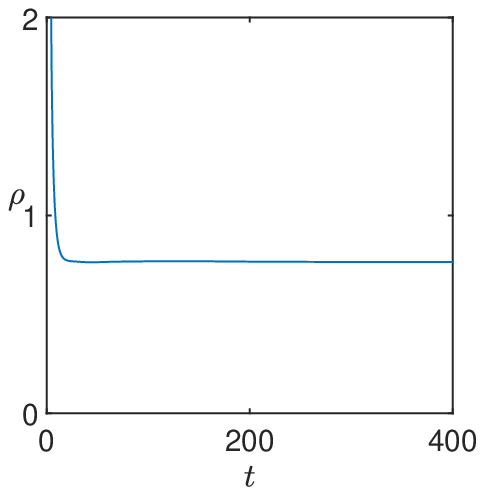}\\
(c) $\rho(t)$
\end{minipage}
\caption{The norm of position, velocity of the first agent and $\rho(t)$ when $\sigma_r=0.3$ and $\sigma_a=1$ }
\label{fig5}
\end{figure}

\begin{figure}[!ht]
\centering
\begin{minipage}{0.45\textwidth}
\centering
\includegraphics[width=0.9\textwidth]{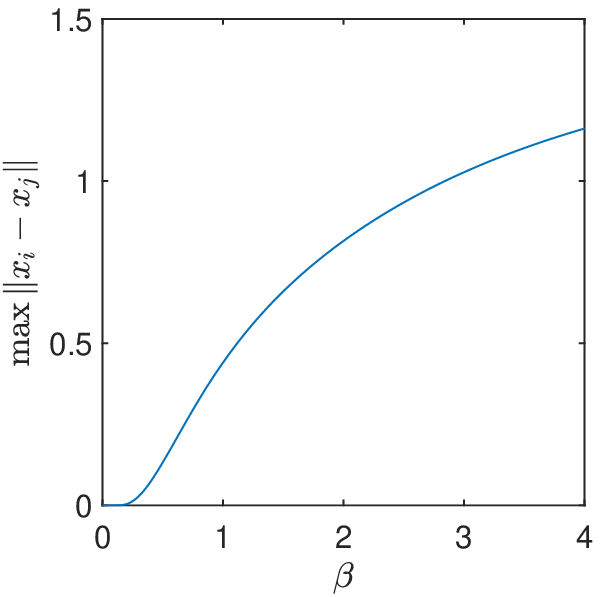}\\
(a) Maximum diameter for \eqref{main}
\vspace{1em}
\end{minipage}
\begin{minipage}{0.45\textwidth}
\centering
\includegraphics[width=0.9\textwidth]{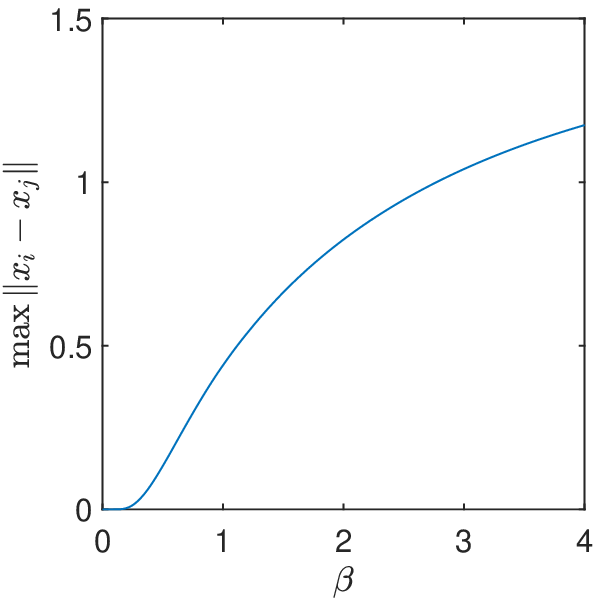}\\
(b) Maximum diameter for \eqref{eqb} with $b=0.2$
\end{minipage}
\caption{Maximum diameter of agents when $\beta\in [0,4]$ and $T=50$}
\label{fig6}
\end{figure}

%
%

Next, we consider more general parameters:
\[\sigma(x)=\sigma_a-\frac{\sigma_r}{x^\beta}.\]
Then the exponent $\beta$ represents the effect of repulsive force.
If $\beta>1$, then the closer the agents are to each other, the stronger the force they repel against each other.
The maximum diameter of agents when $\sigma_a=5$ and $\sigma_r=0.5$ is given in Figure \ref{fig6} and hence we can check that the role of $\beta$.
In addition, in Figure \ref{fig7}, the patterns for the specific values $\beta=0, 0.5, 1, 1.25, 1.26, 1.5, 2, 5$ are displayed. For these cases,  we can find that each agent is arranged at the same  distance as the adjacent agents. It can be seen that their configuration pattern changes for $\beta$ between $1.25$ and $1.26$. See Figure \ref{fig7}(d) and (e).

\begin{figure}[!ht]
\centering
\begin{minipage}{0.243\textwidth}
\centering
\includegraphics[width=0.8\textwidth]{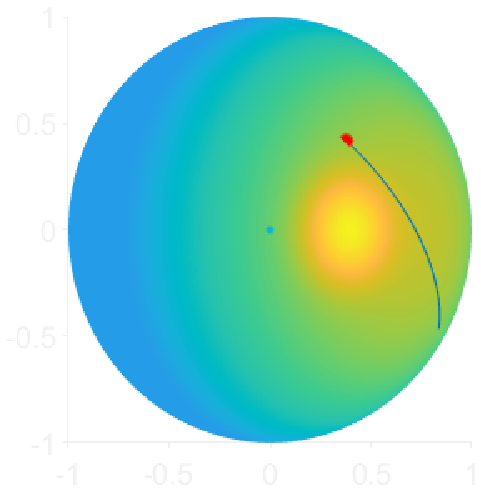}\\
(a) $\beta=0$
\end{minipage}
\begin{minipage}{0.243\textwidth}
\centering
\includegraphics[width=0.8\textwidth]{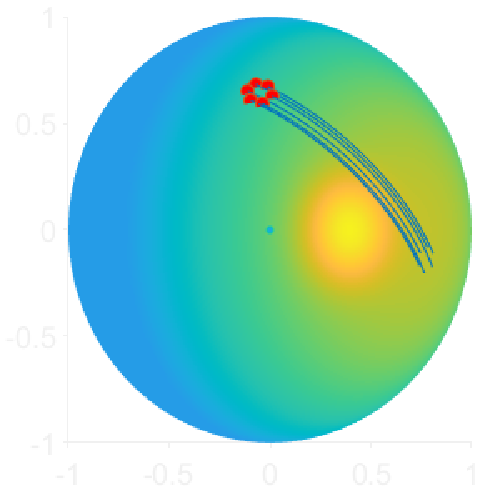}\\
(b) $\beta=0.5$
\end{minipage}
\begin{minipage}{0.243\textwidth}
\centering
\includegraphics[width=0.8\textwidth]{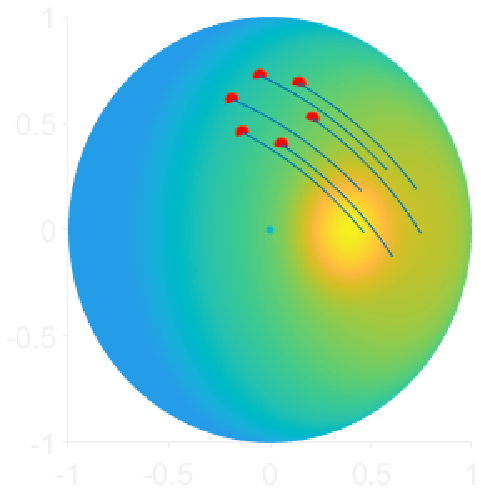}\\
(c) $\beta=1$
\end{minipage}
\begin{minipage}{0.243\textwidth}
\centering
\includegraphics[width=0.8\textwidth]{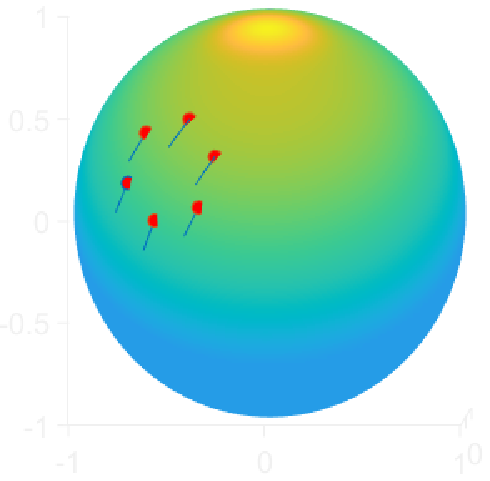}\\
(d) $\beta=1.25$
\end{minipage}
\vfill
\begin{minipage}{0.243\textwidth}
\centering
\includegraphics[width=0.8\textwidth]{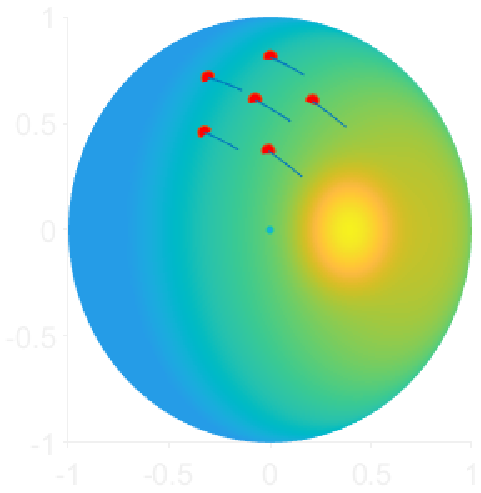}\\
(e) $\beta=1.26$
\end{minipage}
\begin{minipage}{0.243\textwidth}
\centering
\includegraphics[width=0.8\textwidth]{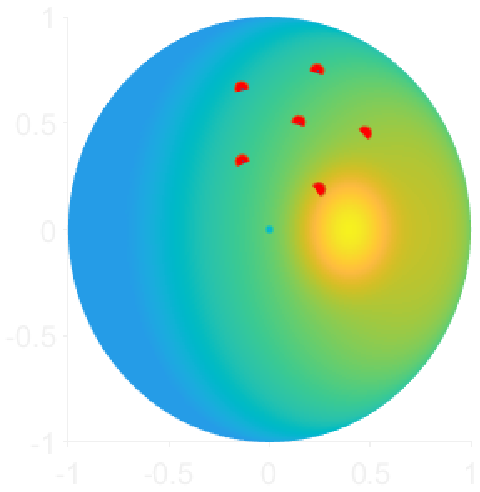}\\
(f) $\beta=1.5$
\end{minipage}
\begin{minipage}{0.243\textwidth}
\centering
\includegraphics[width=0.8\textwidth]{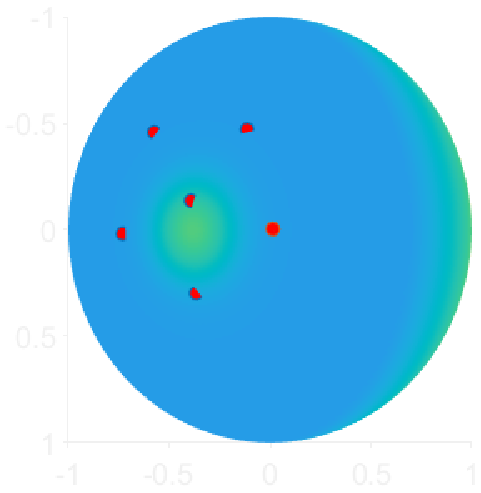}\\
(g) $\beta=2$
\end{minipage}
\begin{minipage}{0.243\textwidth}
\centering
\includegraphics[width=0.8\textwidth]{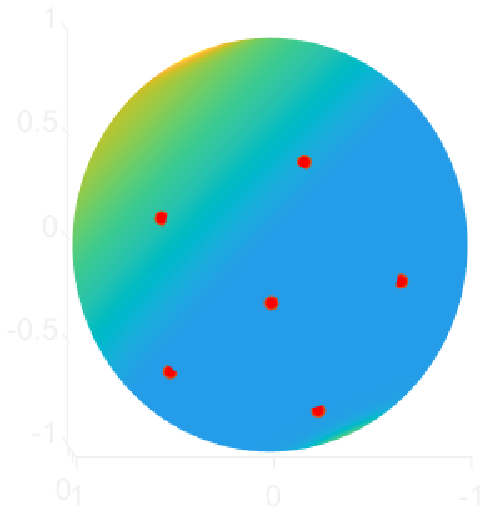}\\
(h) $\beta=5$
\end{minipage}
\caption{Patterns when $\sigma_a=5$ and $\sigma_r=0.5$}
\label{fig7}
\end{figure}

As we can see in the previous examples, the agents could stop depending on their energy. This happens in almost all cases except for a few  limited cases. To avoid this phenomenon, we add a boost term in the second equation of \eqref{main}. The flocking model on a sphere with boost term is given by
\begin{align}
\begin{aligned}\label{eqb}
\dot{x}_i&=v_i,\\
\dot{v}_i&=-\frac{\|v_i\|^2}{\|x_i\|^2}x_i+f_i^bv_i+\sum_{j=1}^N\frac{\psi_{ij}}{N}\big(R_{x_j \shortrightarrow x_i}(v_j)-v_i\big)+\sum_{j=1}^N \frac{\sigma_{ij} }{N}(\|x_i\|^2x_j - \langle x_i,x_j \rangle  x_i),
\end{aligned}
\end{align}
where $f_i^b$ is a boost parameter.
One example of the boost parameter $f_i^b$ is as follows:
\begin{equation}\label{booster}
 f_i^b=\begin{cases}
  \displaystyle
  -\frac{2}{b}\|v_i\|+2, & \text{if $\|v_i(t)\|\leq b$,}
 \\
  0,  & \text{otherwise,}
   \end{cases}
\end{equation}
where $b$ is a positive constant. We note that the boost term increases velocity when the speed falls below a value.

When we consider the boost term such as \eqref{booster}, the agents escape from a sphere after a certain large time because of the accumulation of computation error.
If we want to observe the dynamics of agents for a long time, an additional correction term should be added as in \cite{L-S}.
However, for this boosted case, we observe that the dynamics of the ensemble remains on the sphere. With the same parameter of Figure \ref{fig2}, the time evolution of agents when $b=0.2$ is given in Figure \ref{fig9}.
Comparing to Figure \ref{fig2}, we can check that the agents are moving at the nonzero constant speed without stopping.
See Figure \ref{fig10}.
The velocity is not zero and $\rho(t)$ is perturbed near \[\frac{(N-1)\sigma_r}{2\sigma_a}=1.25\]
 because the boost term pushes the agents steadily.
\begin{figure}[!ht]
\centering
\begin{minipage}{0.243\textwidth}
\centering
\includegraphics[width=\textwidth]{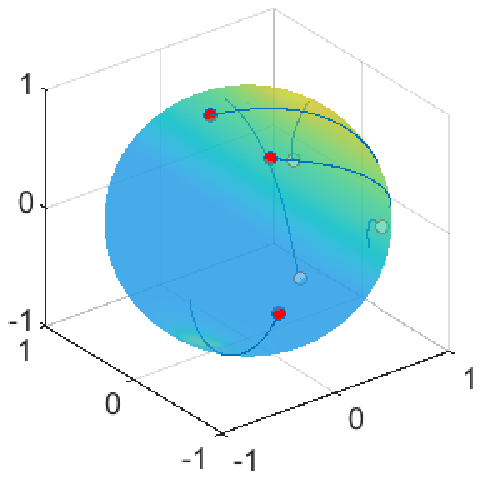}\\
(a) $t=1$
\end{minipage}
\begin{minipage}{0.243\textwidth}
\centering
\includegraphics[width=\textwidth]{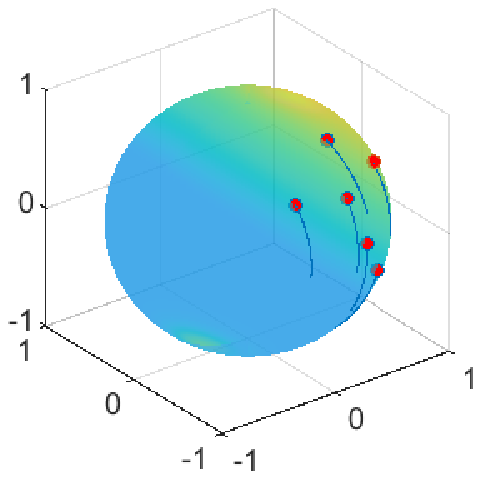}\\
(b) $t=30$
\end{minipage}
\begin{minipage}{0.243\textwidth}
\centering
\includegraphics[width=\textwidth]{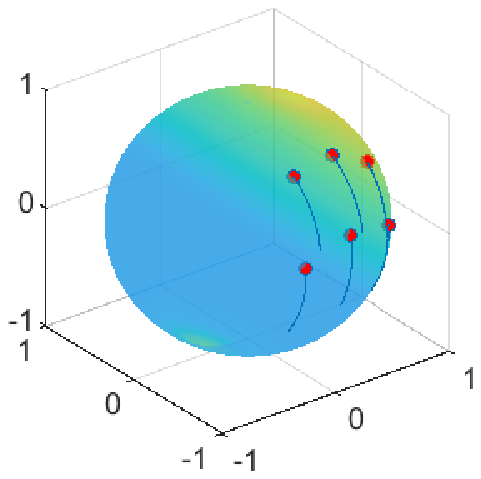}\\
(c) $t=60$
\end{minipage}
\begin{minipage}{0.243\textwidth}
\centering
\includegraphics[width=\textwidth]{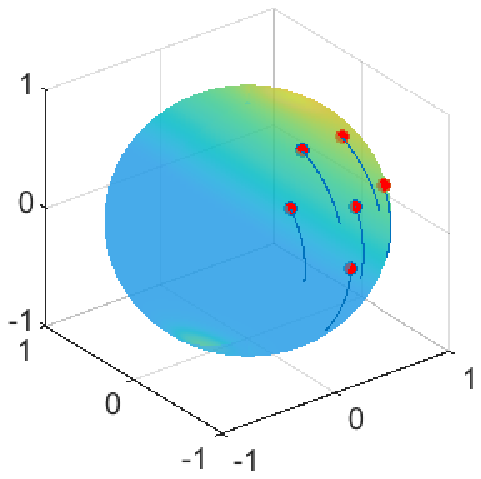}\\
(d) $t=90$
\end{minipage}
\caption{Time evolution of \eqref{eqb} when $\sigma_a=1$ and $\sigma_r=0.5$}
\label{fig9}
\end{figure}

\begin{figure}[!ht]
\centering
\begin{minipage}{0.3\textwidth}
\centering
\includegraphics[width=\textwidth]{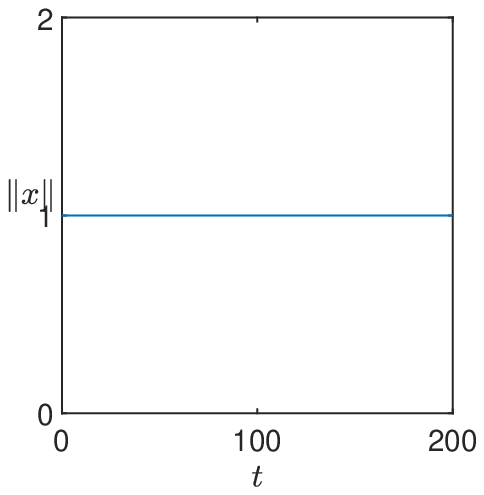}\\
(a) Norm of position $\|x_1\|$
\end{minipage}
\begin{minipage}{0.3\textwidth}
\centering
\includegraphics[width=\textwidth]{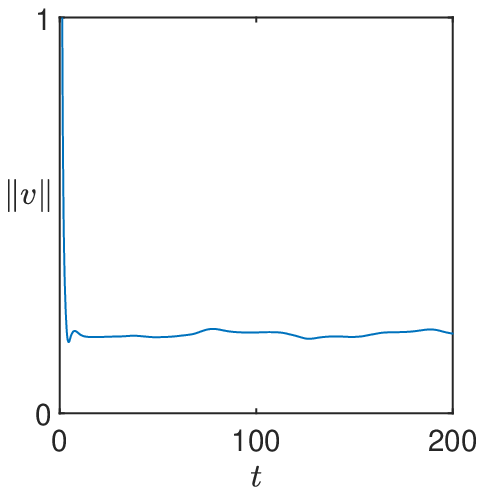}\\
(b) Norm of velocity $\|v_1\|$
\end{minipage}
\begin{minipage}{0.3\textwidth}
\centering
\includegraphics[width=\textwidth]{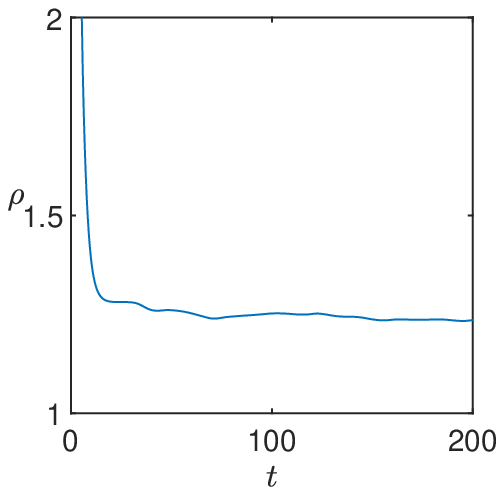}\\
(c) $\rho(t)$
\end{minipage}
\caption{The norm of position, velocity of the first agent and $\rho(t)$ for \eqref{eqb} with $b=0.2$ }
\label{fig10}
\end{figure}

In  \cite{L-S}, the authors consider the following second-order system of ODEs with  the cooperative control law on the sphere: for $1\le i\le N$,
\begin{align}
\begin{aligned}\label{L-S}
\frac{dx_i}{dt} &= v_i, \\
\frac{dv_i}{dt} &= -\|v_i\|^2x_i-k_vv_i+u_i+f_i^0,
\end{aligned}
\end{align}
where
\[u_i=\sum_{j\in \mathcal{N}_i}\omega_{ij}\left(k_a- \frac{k_r}{\|x_i-x_j\|^2}\right)(x_j - \langle x_i,x_j\rangle x_i)\quad \mbox{and} \quad f_i^0=-k_0\left(x_i-\frac{x_i}{\|x_i\|}\right).\]

\begin{figure}[ht!]
\centering
\begin{minipage}{0.243\textwidth}
\centering
\includegraphics[width=\textwidth]{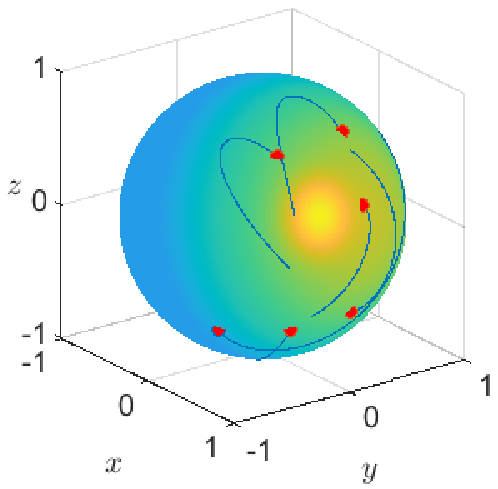}\\
(a) $t=0$
\end{minipage}
\begin{minipage}{0.243\textwidth}
\centering
\includegraphics[width=\textwidth]{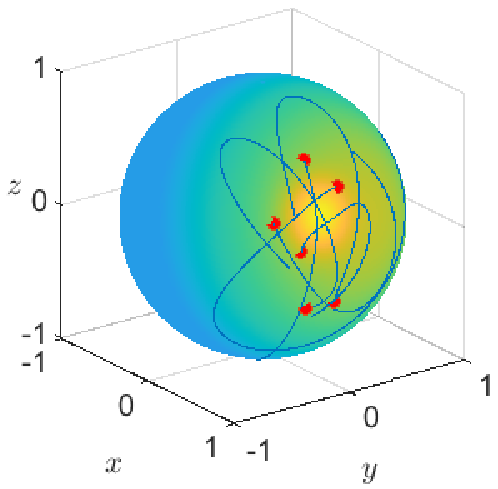}\\
(b) $t=1$
\end{minipage}
\begin{minipage}{0.243\textwidth}
\centering
\includegraphics[width=\textwidth]{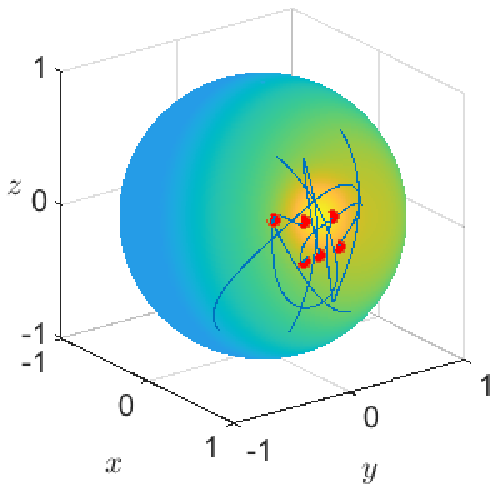}\\
(c) $t=2$
\end{minipage}
\begin{minipage}{0.243\textwidth}
\centering
\includegraphics[width=\textwidth]{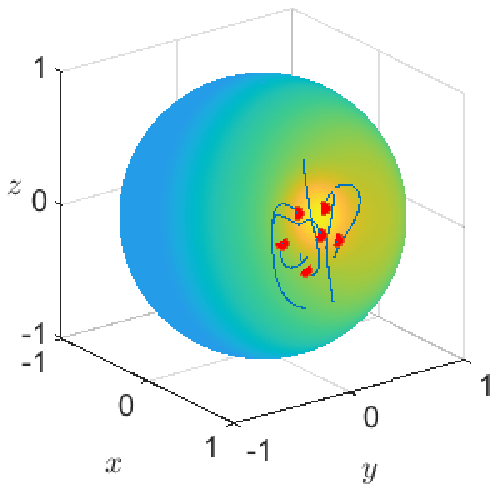}\\
(d) $t=3$
\end{minipage}
\vfill
\begin{minipage}{0.243\textwidth}
\centering
\includegraphics[width=\textwidth]{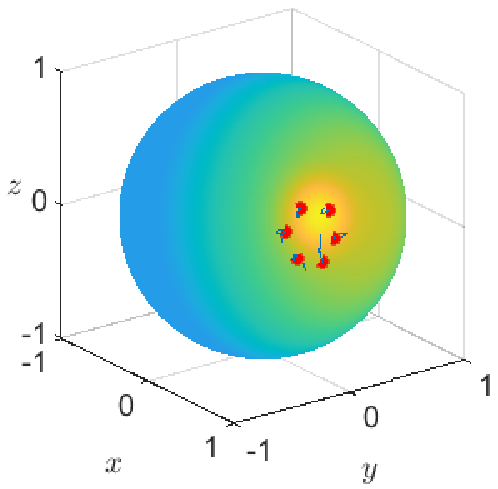}\\
(e) $t=5$
\end{minipage}
\begin{minipage}{0.243\textwidth}
\centering
\includegraphics[width=\textwidth]{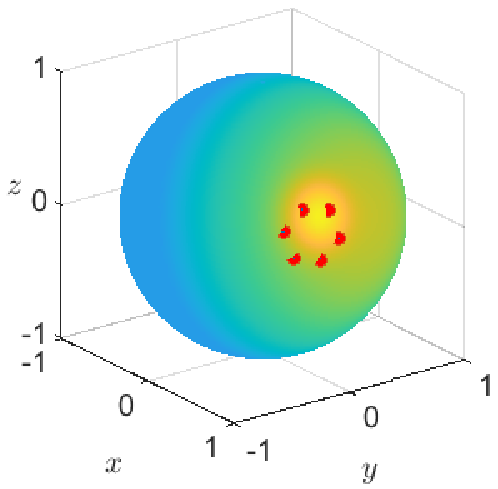}\\
(f) $t=8$
\end{minipage}
\begin{minipage}{0.243\textwidth}
\centering
\includegraphics[width=\textwidth]{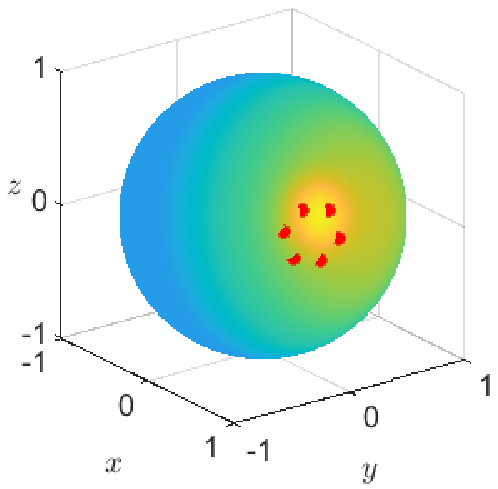}\\
(g) $t=10$
\end{minipage}
\begin{minipage}{0.243\textwidth}
\centering
\includegraphics[width=\textwidth]{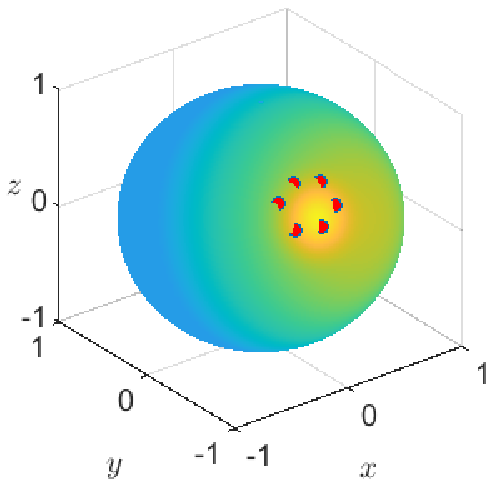}\\
(h) $t=200$
\end{minipage}
\caption{Time evolution of \eqref{L-S}}
\label{fig0}
\end{figure}

\begin{figure}[ht!]
\centering
\begin{minipage}{0.243\textwidth}
\centering
\includegraphics[width=\textwidth]{fig0a.eps}\\
(a) $t=0$
\end{minipage}
\begin{minipage}{0.243\textwidth}
\centering
\includegraphics[width=\textwidth]{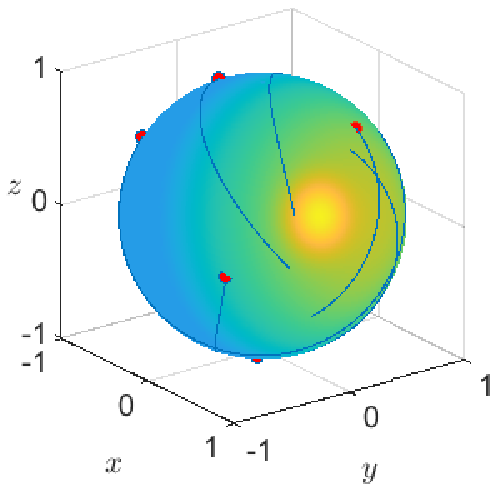}\\
(b) $t=1$
\end{minipage}
\begin{minipage}{0.243\textwidth}
\centering
\includegraphics[width=\textwidth]{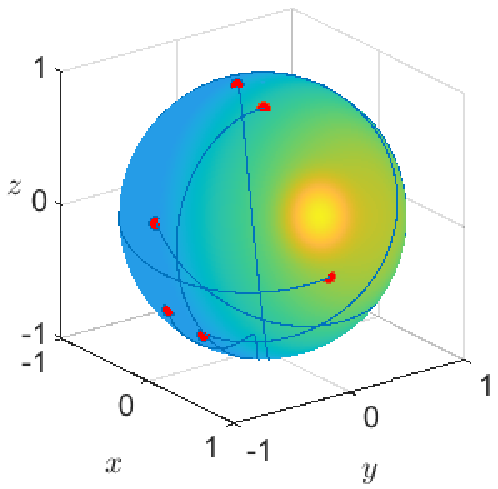}\\
(c) $t=2$
\end{minipage}
\begin{minipage}{0.243\textwidth}
\centering
\includegraphics[width=\textwidth]{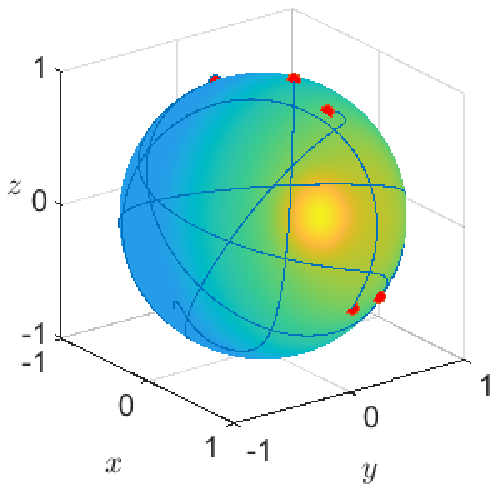}\\
(d) $t=3$
\end{minipage}
\vfill
\begin{minipage}{0.243\textwidth}
\centering
\includegraphics[width=\textwidth]{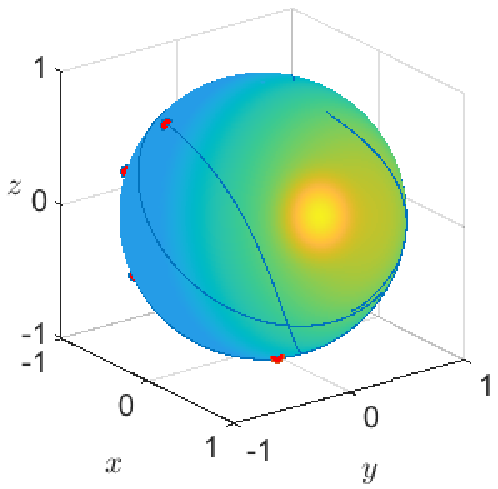}\\
(e) $t=5$
\end{minipage}
\begin{minipage}{0.243\textwidth}
\centering
\includegraphics[width=\textwidth]{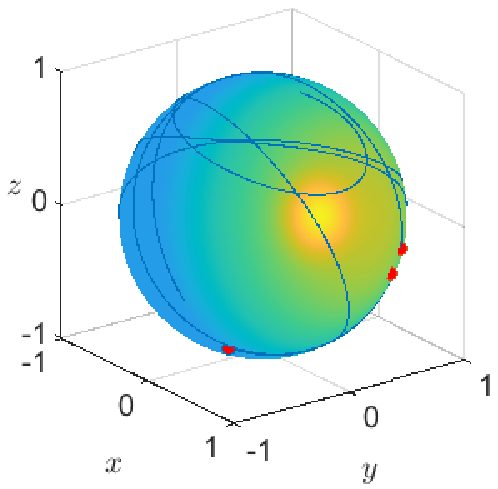}\\
(f) $t=8$
\end{minipage}
\begin{minipage}{0.243\textwidth}
\centering
\includegraphics[width=\textwidth]{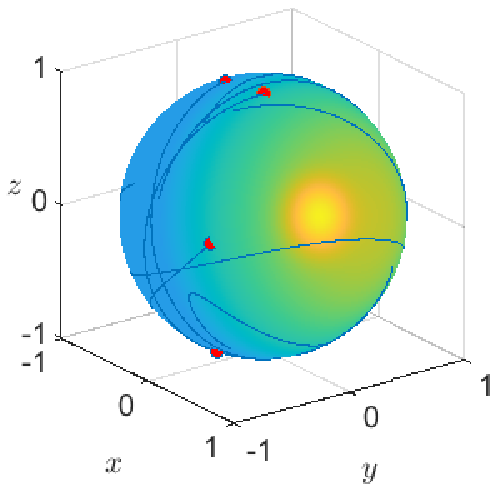}\\
(g) $t=10$
\end{minipage}
\begin{minipage}{0.243\textwidth}
\centering
\includegraphics[width=\textwidth]{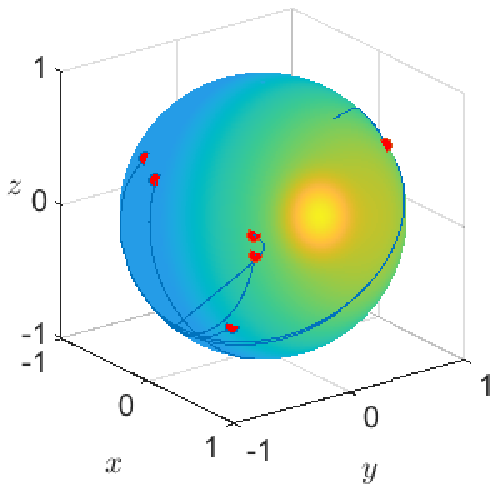}\\
(h) $t=200$
\end{minipage}
\caption{Time evolution of \eqref{L-S}when $k_v=0$ and $f_i^b$ is equipped}
\label{fig11}
\end{figure}
Here, agents are controlled by the control law $u_i$ and stayed on the sphere by feedback term $f_i^0$.
The parameters $k_v$, $k_a$, and $k_r$ are a damping, attraction, and repulsion constant, respectively. Due to the damping term $k_vv_i$, the ensemble of all agents is stabilized and one can obtain the stationary formation on the sphere.
Therefore, this model is appropriate for the case when the target area is fixed.
See Figure \ref{fig0}.
This figure is the time evolution of \eqref{L-S} with parameters $(N,k_a,k_r,k_v,k_0)=(6,1,0.1,7,10^4)$ and initial configuration is randomly chosen satisfying $\|x_i(0)\|=1$ and $\langle x_i(0),v_i(0)\rangle=0$  for all index $i\in \{1,\ldots,N\}$.
Note that if there is no damping term, that is, $k_v=0$, then the system \eqref{L-S} yields chaotically moving agents.
If the boost term $f_i^b$ is added in the absence of damping, agents move more chaotically.
See Figure \ref{fig11}. However, for the case of  the system of \eqref{main}, even if there is no damping term or there is a boost term, we can obtain a moving formation due to the flocking operator in the model \eqref{main}. Therefore, it is suitable to apply this model for detection or surveillance problems where the target area is not decided or the target area is relatively large.

%

\section{Conclusion}\label{sec5}\setcounter{equation}{0}

In this paper, we studied a spherical flocking model with attractive and repulsive forces. For any admissible initial conditions,  we demonstrate the velocity alignment as well as the global well-posedness of the model in Theorem \ref{thm:1}. In Theorem \ref{thm:2}, we classify all possible asymptotic configurations for this model: rendezvous, formation configuration, and uniform deployment. In particular, our model maintains the desired pattern of a formation flight with nonzero constant speed. We observe that this nonzero speed formation cannot be achieved without the flocking operator, which acts as a stabilizer. Our analytic results were supported by numerical simulations.

\bibliographystyle{amsplain}

\end{document}